\documentclass[reqno]{amsart}
\usepackage[utf8]{inputenc}

\title{Stability in affine logic}
\author{Itaï Ben Yaacov}
\address{
  Universit\'e Claude Bernard Lyon 1 \\
  Institut Camille Jordan \\
  Lyon \\
  France}
\urladdr{\url{https://math.univ-lyon1.fr/~begnac/}}

\author{Tomás Ibarlucía}
\address{Universit\'e Paris Cit\'e \\
  CNRS \\
  IMJ-PRG \\
  F-75006 Paris \\
  France.}
\urladdr{\url{https://webusers.imj-prg.fr/~tomas.ibarlucia}}

\usepackage{macros}
\usepackage{eucal} 

\usepackage{tikz-cd}
\usepackage[displaymath,textmath,sections,graphics]{preview}
\PreviewEnvironment{tikzcd}

\newcommand{\Rand}{\mathrm{R}}
\newcommand{\Bau}{\mathrm{Bau}}
\newcommand{\Mor}{\mathrm{Mor}}
\newcommand{\ext}{\mathrm{ext}}

\newcommand{\crc}{\mathrm{cr}}

\newcommand{\PMP}{\mathrm{PMP}}
\newcommand{\FPMP}{\mathrm{FPMP}}

\usepackage{enumitem}

\numberwithin{equation}{section}

\begin{document}

\begin{abstract}
  We develop foundational aspects of stability theory in affine logic.
  On the one hand, we prove appropriate affine versions of many classical results, including definability of types, existence of non-forking extensions, and other fundamental properties of forking calculus.
  Most notably, stationarity holds over arbitrary sets (in fact, every type is Lascar strong).
  On the other hand, we prove that stability is preserved under direct integrals of measurable fields of structures.
  We deduce that stability in the extremal models of an affine theory implies stability of the theory.
  We also deduce that the affine part of a stable continuous logic theory is affinely stable, generalising the result of preservation of stability under randomisations.
\end{abstract}

\maketitle

\tableofcontents

\section*{Introduction}

First-order \emph{affine logic} can be defined as the fragment of first-order continuous logic obtained by restricting the connectives to affine functions.
This was originally introduced by Bagheri as \emph{linear logic} in a series of papers \cite{Bagheri2010, Bagheri2014,BagheriSafari2014,Bagheri2021}, and was more recently the subject of a comprehensive study by Tsankov and the authors in \cite{BenYaacov-Ibarlucia-Tsankov:AffineLogic}.
The present paper is a continuation of the latter.
In particular, we shall follow its terminology and notation.

In \cite{BenYaacov-Ibarlucia-Tsankov:AffineLogic} we were concerned with the general foundational aspects of affine logic.
These being laid down, we initiate here the study of model-theoretic stability in this setting.
We start with the consideration of \emph{local stability}, namely stability of a single formula, inside a single structure.
In the context of continuous logic, local stability was introduced and studied, adapting classical model-theoretic techniques, in \cite{BenYaacov-Usvyatsov:CFO}.
However, the underlying definition goes back as far as Krivine and Maurey \cite{Krivine-Maurey:EspacesDeBanachStables}, or even Grothendieck \cite{Grothendieck:CriteresDeCompacite}.
This last observation was taken up in \cite{BenYaacov:Grothendieck}, where the goal was to develop the basic properties of stability (in particular, the definability of types) as quickly as possible using tools of functional analysis and Banach space theory.

In addition to its fundamental place in modern model theory, local stability is at the heart of many recent applications of the field to other areas.
Let us mention, for instance, \cite{Hrushovski:ApproximateSubgroups, BenYaacov-Tsankov:WAP, IbarluciaTsankov:StronglyErgodic, Ibarlucia:InfiniteDimensional, ChevalierHrushovski, ChavarriaConantPillay}.

Since the definitions of stability in a structure for a single affine or continuous formula are identical, and in fact have nothing to do with logic, we opted here for an abstract functional-analytic approach for the basic setup.
This is done in \autoref{sec:DoubleLimit}.
For the benefit of readers who are not model theorists, we present it in a manner that avoids any model-theoretic notions (local type spaces are introduced through the back door, as $\overline{\varphi_A}$, $\cco(\varphi_A)$, and their counterparts for $B$).
For the benefit of readers who are model theorists, and in contrast with \cite{BenYaacov:Grothendieck}, we keep this presentation as elementary and as self-contained as possible.
In particular, we make no direct use of Grothendieck \cite{Grothendieck:CriteresDeCompacite}, and replace the Eberlein--Šmulian Theorem with a Downward Löwenheim--Skolem argument in \autoref{lem:MeanDefinableLS}.

In \autoref{sec:StabilityInStructure} we translate the contents of \autoref{sec:DoubleLimit} to local stability for an affine formula in a structure.
Most importantly, in \autoref{thm:DoubleLimitFormula}, which translates \autoref{thm:DoubleLimitTopological}, we show that for an affine formula $\varphi(x,y)$ and a structure $M$, the following are equivalent:
\begin{enumerate}
\item The formula $\varphi$ is stable in $M$.
\item All affine $\varphi$-types $p$ and $\varphi^\op$-types $q$ over $M$ are \emph{affinely definable}, and moreover $dp(q) = dq(p)$.
\item All affine $\varphi$-types and $\varphi^\op$-types over $M$ are \emph{mean-definable} (in which case the additional symmetry property holds automatically).
\end{enumerate}
In fact, the same is proved for types and formulas in continuous logic as well, given that the topological contents is (almost) identical and was covered in \autoref{sec:DoubleLimit}.

\autoref{sec:DirectIntegrals} is dedicated to showing that stability in a structure is preserved under direct integrals (\autoref{thm:stability-direct-integral}).
The direct integral of measurable fields of structures is, in a sense, the fundamental construction of affine logic, distinguishing it from continuous logic, and plays a key role throughout \cite{BenYaacov-Ibarlucia-Tsankov:AffineLogic}.
While the direct integral is conceptually related to the ultraproduct construction, it is quite complementary to it, and preservation under direct integrals should be contrasted with the fact that stability in a structure is \emph{not} preserved under ultraproducts.

In \autoref{sec:StabilityInTheory} we consider stability of an affine formula $\varphi(x,y)$ in an affine theory $T$, i.e., in all of its models simultaneously.
Under this hypothesis we may use ultraproducts, and thus improve earlier results.
For example, \autoref{thm:StableMeanDefinableUniformSimultaneous} is a \emph{uniform} mean-definability result for types.
We prove the equivalence of stability in a theory with several other conditions in \autoref{thm:stable-in-T-equivalences}.
Some are ``standard'', such as definability of and counting (affine) types, and some are specific to affine logic.
Most importantly, we use \autoref{thm:stability-direct-integral} to show that for $\varphi$ to be stable in $T$, it is enough for $\varphi$ to be stable in the extremal models of $T$.
This latter equivalence has several nice corollaries.
In particular, it implies that if an affine formula is stable in a continuous theory $\bT$, then it is also stable in its affine part $\bT_\aff$.
We leave open the question of finding a more direct argument proving this.

In \autoref{sec:NonForking} we define \emph{non-forking extensions} of types in a stable formula, or in several such, first to a model and then between arbitrary sets.
Unlike the situation in classical or continuous logic, the non-forking extension is always unique, that is to say that types over arbitrary sets are stationary.
This notion of non-forking extensions allows us to define in \autoref{sec:Independence} a notion of independence relative to all stable formulas in a theory.
As in continuous logic, if the theory is affinely stable, then this notion satisfies all the usual properties of a ``stable notion of independence'', and conversely, if such a notion exists, then it must agree with affine stable independence (and the theory is affinely stable).

We conclude with \autoref{sec:LascarType}, which, strictly speaking, has nothing to do with stability, showing that types over arbitrary sets, in the sense of affine logic, coincide with Lascar types.
This explains, in a sense, the stationarity phenomenon observed in \autoref{sec:NonForking}.
We deduce the same for theories in continuous logic that are of the form $T_\crc$, namely, an affine theory $T$ together with the \emph{convex realisation} axiom scheme.
This generalises (and slightly improves) a theorem from \cite{BenYaacov:RandomVariables}, asserting that in a randomisation of a continuous theories, types over sets are Lascar strong.

\section{The double limit property and mean-definability}
\label{sec:DoubleLimit}

We recall that if $E$ is a locally convex Hausdorff topological vector space and $X \subseteq E$ is compact and convex, then $X$, equipped with the topology and the structure of convex combinations, is a \emph{compact convex set}.
If $A \subseteq X$, then $\co(A)$ is the collection of convex combinations of points in $A$, and $\cco(A)$ is its closure.
Alternatively, by the Hahn-Banach Theorem, $\cco(A)$ coincides with the intersection of all closed half-spaces $\bigl\{x : \lambda(x) \leq \alpha \bigr\}$, where $\lambda$ is a continuous affine real-valued function on $X$ (or on $E$) and $\lambda\rest_A \leq \alpha$.
In the present section we are mostly interested in the following archetypical examples.

\begin{exm}
  \label{exm:PowerR}
  Given a set $B$, we equip the space $E = \bR^B$ with the product topology.
  Compact convex subsets of $E$ will be mostly obtained as closed convex subsets of $\bI^B \subseteq E$, for some compact interval $\bI \subseteq \bR$.
  The topological dual $E^*$ consists of all $\lambda(x) = \sum_i \lambda_i x_i$, with all but finitely many coefficients vanishing.
\end{exm}

\begin{exm}
  \label{exm:CX}
  When $X$ is a compact Hausdorff space, we denote by $C(X)$ the Banach space of continuous real-valued functions on $X$, equipped with the supremum norm.
  By the Riesz Representation Theorem, $C(X)^*$ is the space of signed regular Borel measures on $X$.
  An important convex subset of the latter is $\cM(X)$, the space of regular Borel probability measures.
  Equipping $C(X)^*$ with the weak$^*$ topology, $\cM(X)$ is a compact convex set, whose set of extreme points is homeomorphic to $X$ (identifying $x \in X$ with the corresponding Dirac measure).
\end{exm}

\begin{exm}
  \label{exm:AX}
  When $X$ is a compact convex set, we denote by $\cA(X) \subseteq C(X)$ the Banach subspace of continuous affine real-valued functions on $X$.
  Let $\pi\colon C(X)^* \rightarrow \cA(X)^*$ denote the corresponding quotient map, namely, the restriction of (evaluations of) measures to affine functions.

  For $x \in X$, define $\hat{x} \in \cA(X)^*$ by $\hat{x}(f) = f(x)$.
  The map $x \mapsto \hat{x}$ is an affine topological embedding of $X$ in $E = \cA(X)^*$.
  If $\mu \in \cM(X)$ and $x = R(\mu) \in X$ is its barycentre, then $\pi(\mu) = \hat{x}$.
  Therefore, since $\cM(X)$ generates $C(X)^*$, the set $\{\hat{x} : x \in X\}$ generates $\cA(X)^*$.
\end{exm}

Throughout, we fix a compact interval $\bI \subseteq \bR$.

\begin{lem}
  \label{lem:SeparatelyContinuousExtensionToM}
  Let $X$ be a compact Hausdorff space, $Y$ a metrisable space, and $\varphi\colon X \times Y \rightarrow \bI$ a separately continuous function.
  Then, under the identification $X \subseteq \cM(X)$, $\varphi$ admits a unique extension to a separately continuous function $\cM(X) \times Y \rightarrow \bI$, affine in the first argument.
  If $Y$ is convex and $\varphi$ is affine in the second argument, then so is its extension.
\end{lem}
\begin{proof}
  Uniqueness is clear.
  For existence, consider the function $(\mu,y) \mapsto \int \varphi(x,y) \, d\mu(x)$.
  This is affine in $\mu$, and continuous in $\mu$ by definition of the topology on $\cM(X)$.
  Since $\varphi$ is bounded and $Y$ is metrisable, by the Dominated Convergence Theorem this function is also continuous in $y$.
  Preservation of affineness on $Y$ is clear.
\end{proof}

\begin{lem}
  \label{lem:SeparatelyContinuousSeparatelyAffineWeakTopology}
  Let $X$ and $Y$ be compact convex sets, and let $\varphi\colon X \times Y \rightarrow \bI$ be separately continuous and separately affine.
  For $x \in X$ and $y \in Y$, define $\varphi_x(y) = \varphi(x,y)$.
  Then $\varphi_x \in \cA(Y)$, and the map $x \mapsto \varphi_x$ is affine and continuous in the weak topology on $\cA(Y)$.
\end{lem}
\begin{proof}
  It is clear that $\varphi_x \in \cA(Y)$ and that $x \mapsto \varphi_x$ is affine.
  If $y \in Y$, then $x \mapsto \varphi(x,y) = \hat{y}(\varphi_x)$ is continuous in $x$.
  Since $\{\hat{y} : y \in Y\}$ generates $\cA(Y)^*$, the map $x \mapsto \varphi_x$ is continuous in the weak topology on $\cA(Y)$.
\end{proof}

\begin{lem}
  \label{lem:Mazur}
  In a Banach space, the weak closure of a convex set agrees with its norm closure.
\end{lem}
\begin{proof}
  Immediate from the Hahn-Banach Separation Theorem.
\end{proof}

In what follows, $A$ and $B$ are non-empty sets and $\varphi\colon A \times B \rightarrow \bI$ a function.
As in \autoref{exm:PowerR}, we equip $\bI^B$ with the product topology, namely the topology of pointwise convergence.
For $a \in A$ we define $\varphi_a \in \bI^B$ by $\varphi_a(b) = \varphi(a,b)$, and $\varphi_A = \{\varphi_a : a \in A\} \subseteq \bI^B$.
We have a natural map
\begin{gather*}
\pi_A\colon A \rightarrow \varphi_A \subseteq \overline{\varphi_A} \subseteq \cco(\varphi_A) \subseteq \bI^B,
\end{gather*}
sending $a \mapsto \varphi_a$.
Similarly for $\varphi^b(a) = \varphi(a,b)$ and $\pi_B\colon B \rightarrow \varphi^B \subseteq \overline{\varphi^B} \subseteq \cco(\varphi^B) \subseteq \bI^A$.

At times we may also consider the metric of uniform convergence on $\bI^B$, for which we use norm notation: $\|f - g\|_B = \sup_{b \in B} \, |f(b) - g(b)|$.

\begin{exm}
  If $M$ is an $\cL$-structure, and $\varphi(x,y)$ an $\cL$-formula, in continuous or affine logic, then we may take $A = M^x$, $B = M^y$, and by a common abuse of notation, identify $\varphi$ with its interpretation in $M$.

  Then we may identify $\overline{\varphi_A}$ (respectively, $\cco(\varphi_A)$) with the space of continuous (respectively, affine) $\varphi$-types over $M$ (see \autoref{sec:StabilityInStructure}).
\end{exm}

\begin{dfn}
  \label{dfn:MeanDefinable}
  Let $A$ and $B$ be non-empty sets, and $\varphi\colon A \times B \rightarrow \bI$.
  Say that $p \in \bI^B$ is \emph{$\varphi$-mean-definable} if it belongs to the uniform closure of $\co(\varphi_A)$, i.e., if for every $\varepsilon > 0$, there exists a finite sequence $(a_i : i < n) \subseteq A$ such that
  \begin{gather*}
    \left\| p - \frac{1}{n} \sum_{i<n} \varphi_{a_i} \right\|_B
    = \sup_{b \in B} \ \left| p(b) - \frac{1}{n} \sum_{i<n} \varphi(a_i,b) \right|
    \leq \varepsilon.
  \end{gather*}
\end{dfn}

\begin{lem}
  \label{lem:MeanDefinableLS}
  Let $A$ and $B$ be non-empty sets, and $\varphi\colon A \times B \rightarrow \bI$.
  Let $p \in \bI^B$.
  Then there exist countable non-empty subsets $C \subseteq A$ and $D \subseteq B$ such that if $\eta\colon A \rightarrow \bI$ is a linear combination of functions $\varphi^d$, $d \in D$, and $\theta\colon B \rightarrow \bI$ is a convex combination of functions $\varphi_c$, $c \in C$, both with rational coefficients, then
  \begin{gather}
    \label{eq:MeanDefinableLS}
    \sup_A \eta = \sup_C \eta,
    \qquad
    \|p-\theta\|_B = \|p-\theta\|_D.
  \end{gather}
  Letting $q = p\rest_D \in \bI^D$ and $\psi = \varphi\rest_{C \times D} \colon C \times D \rightarrow \bI$, this implies that
  \begin{itemize}
  \item If $p \in \cco(\varphi_A)$, then $q \in \cco(\psi_C)$.
  \item If $q$ is $\psi$-mean-definable, then $p$ is $\varphi$-mean-definable.
  \end{itemize}
\end{lem}
\begin{proof}
  The first assertion is a restricted instance of the Downward Löwenheim--Skolem Theorem, proved by adding witnesses to $C$ and $D$ until \autoref{eq:MeanDefinableLS} holds.
  For the second, $\varphi$-mean-definability is fairly immediate.
  Finally, if $q \notin \cco(\psi_C)$, then there exists a continuous linear functional $\lambda\colon \bR^D \rightarrow \bR$ that separates $q$ from $\psi_C$.
  As in \autoref{exm:PowerR}, we may write $\lambda(x) = \sum_d \lambda_d x_d$, and assume that the $\lambda_d$ are rational.
  Considering the first equality of \autoref{eq:MeanDefinableLS} for $\eta = \sum_d \lambda_d \varphi^d$, we see that the natural extension of $\lambda$ to $\bI^B$ witnesses that $p \notin \cco(\varphi_A)$.
\end{proof}

\begin{dfn}
  \label{dfn:DoubleLimit}
  Let $A$ and $B$ be sets, and $\varphi\colon A \times B \rightarrow \bI$.
  We say that $\varphi$ has the \emph{double limit property} if for any two sequences $(a_n)$ in $A$ and $(b_m)$ in $B$, we have
  \begin{gather}
    \label{eq:DoubleLimit}
    \lim_n \lim_m \varphi(a_n,b_m) = \lim_m \lim_n \varphi(a_n,b_m)
  \end{gather}
  as soon as all the limits exist.
\end{dfn}

Given $\varphi\colon A \times B \rightarrow \bI$, we define $\varphi^\op\colon B \times A \rightarrow \bI$ by $\varphi^\op(b,a) = \varphi(a,b)$, observing that the double limit property is equivalent for $\varphi$ and for $\varphi^\op$.
If $\varphi_a = \varphi_c$ and $\varphi^b = \varphi^d$, then
\begin{gather*}
  \varphi(a,b) = \varphi_a(b) = \varphi_c(b) = \varphi^b(c) = \varphi^d(c) = \varphi(c,d).
\end{gather*}
It is therefore legitimate to define $\overline{\varphi}(\varphi_a,\varphi^b) = \varphi(a,b)$, and then $\varphi = \overline{\varphi} \circ (\pi_A \times \pi_B)$.
By definition, $\overline{\varphi}$ is separately continuous on $\varphi_A \times \varphi^B$.

\begin{thm}
  \label{thm:DoubleLimitTopological}
  Let $A$ and $B$ be non-empty sets, and $\varphi\colon A \times B \rightarrow \bI$.
  Then the following are equivalent:
  \begin{enumerate}
  \item
    \label{item:DoubleLimit}
    The function $\varphi$ has the double limit property.
  \item
    \label{item:DoubleLimitClosure}
    The function $\overline{\varphi}\colon \varphi_A \times \varphi^B \rightarrow \bI$ extends to a separately continuous function $\overline{\varphi_A} \times \overline{\varphi^B} \rightarrow \bI$.
  \item
    \label{item:DoubleLimitCCO}
    The function $\overline{\varphi}\colon \varphi_A \times \varphi^B \rightarrow \bI$ extends to a separately continuous, separately affine function $\cco(\varphi_A) \times \cco(\varphi^B) \rightarrow \bI$.
  \item
    \label{item:DoubleLimitDefinableClosure}
    Every $p \in \overline{\varphi_A}$ is $\varphi$-mean-definable, and every $q \in \overline{\varphi^B}$ is $\varphi^\op$-mean-definable.
  \item
    \label{item:DoubleLimitDefinableCCO}
    Every $p \in \cco(\varphi_A)$ is $\varphi$-mean-definable, and every $q \in \cco(\varphi^B)$ is $\varphi^\op$-mean-definable.
  \end{enumerate}
  Moreover, the extensions in \autoref{item:DoubleLimitClosure} and \autoref{item:DoubleLimitCCO} are unique.
\end{thm}
\begin{proof}
  The uniqueness part is clear.
  For the main assertion, let us consider first the case where $A$ and $B$ are countable, so $\bI^A$ and $\bI^B$ are metrisable.
  Then we have the following implications.
  \begin{cycprf*}[1]
  \item
    Let $p \in \overline{\varphi_A}$, and choose a sequence $(a_n)$ in $A$ such that $\varphi_{a_n} \rightarrow p$ pointwise.
    Similarly, for a point $q \in \overline{\varphi^B}$ we may choose a sequence $\varphi^{b_m} \rightarrow q$.
    Possibly passing to subsequences, we may assume that all the limits in \autoref{eq:DoubleLimit} exist, so
    \begin{gather}
      \label{eq:DoubleLimitTypes}
      \lim_n q(a_n) = \lim_m p(b_m).
    \end{gather}
    The left hand side does not depend on $(b_m)$ and the right hand side does not depend on $(a_n)$.
    Therefore, equality holds regardless of the chosen sequences.
    Call the common value $\varphi_1(p,q)$.
    By construction, $\varphi_1\colon \overline{\varphi_A} \times \overline{\varphi^B} \rightarrow \bI$ is separately continuous.
    In addition, $\varphi_1(p,\varphi^b) = p(b)$, so $\varphi_1(\varphi_a,\varphi^b) = \varphi(a,b)$, as desired.
  \item
    Let us use $\overline{\varphi}$ to denote the unique separately continuous extension to $\overline{\varphi_A} \times \overline{\varphi^B}$.
    Therefore, by \autoref{lem:SeparatelyContinuousExtensionToM}, it extends in a unique fashion to a separately affine, separately continuous function $\varphi_2\colon \cM(\overline{\varphi_A}) \times \cM(\overline{\varphi^B}) \rightarrow \bI$ (we use the fact that $\cM(\overline{\varphi_A})$ is separable as well).
    Assume that $p \in \cco(\varphi_A)$ is the barycentre of $\mu \in \cM(\overline{\varphi_A})$.
    If $\varphi^{b_m} \rightarrow q$, then
    \begin{gather*}
      \varphi_2(\mu,q) = \lim_m \int \overline{\varphi}(r,\varphi^{b_m}) \, d\mu(r) = \lim_m \int r(b_m) \, d\mu(r) = \lim_m p(b_m).
    \end{gather*}
    Therefore $\varphi_2(\mu,q)$ only depends on $(p,q)$.
    Consequently, $\varphi_2(\mu,\nu)$ only depends on $(p,\nu)$.
    By the same logic, if $q$ is the barycentre of $\nu$, then $\varphi_2(\mu,\nu)$ only depends on $(p,q)$.
    Since $\cM(\overline{\varphi_A}) \rightarrow \cco(\varphi_A)$ and $\cM(\overline{\varphi^B}) \rightarrow \cco(\varphi^B)$ are affine topological quotient maps, $\varphi_2$ factors through a separately affine, separately continuous function $\cco(\varphi_A) \times \cco(\varphi^B) \rightarrow \bI$.
  \item[\impnum{3}{5}]
    Let us use $\overline{\varphi}$ to denote the unique affine separately continuous extension to $\cco(\varphi_A) \times \cco(\varphi^B)$ and let $E = \cA\bigl( \cco(\varphi^B) \bigr)$.
    Let $\iota\colon \cco(\varphi_A) \rightarrow E$ be as given by \autoref{lem:SeparatelyContinuousSeparatelyAffineWeakTopology}.
    If $p \in \cco(\varphi_A)$, then $\iota(p)$ belongs to the weak closure of $\co\bigl( \iota(\varphi_a) : a \in A \bigl)$, and therefore to the norm closure by \autoref{lem:Mazur}.
    Thus $p$ is $\varphi$-mean-definable.

    Similarly, every $q \in \overline{\varphi^B}$ is $\varphi^\op$-mean-definable.
  \end{cycprf*}

  \smallskip
  Now we consider the general case.
  \begin{cycprf}[1]
  \item[\impnum{1}{5}]
    By \autoref{lem:MeanDefinableLS}, it is enough to consider the case where $A$ and $B$ are countable, which we have already established.
  \item[\impnum{5}{4}, \impnum{3}{2}]
    Immediate.
  \item[\impnum{5}{3}, \impnum{4}{2}]
    Let $\co(A)$ denote the collection of formal convex combinations of members of $A$.
    To $a = \sum \lambda_i a_i \in \co(A)$ we can associate the function $\varphi_a = \sum \lambda_i \varphi_{a_i} \in \co(\varphi_A) \subseteq \bI^B$, as well as $\hat{a} \in \cA\bigl( \cco(\varphi^B)\bigr)$, defined by $\hat{a}(q) = \sum \lambda_i q(a_i)$.
    If $a,a' \in \co(A)$, then
    \begin{gather*}
      \|\varphi_a - \varphi_{a'}\|_B
      = \|\hat{a} - \hat{a}'\|_{\varphi^B}
      = \|\hat{a} - \hat{a}'\|_{\co(\varphi^B)}
      = \|\hat{a} - \hat{a}'\|_{\cco(\varphi^B)}.
    \end{gather*}
    If $p \in \cco(\varphi_A)$ is $\varphi$-mean-definable, then there exists a sequence $(a_n) \subseteq \co(A)$ such that $\varphi_{a_n} \rightarrow p$ uniformly in $\bI^B$, so $(\hat{a}_n) \subseteq \cA\bigl( \cco(\varphi^B) \bigr)$ is Cauchy.
    Its limit only depends on $p$, so let $\hat{a}_n \rightarrow \hat{p}$.
    Similarly, if $q \in \cco(\varphi^B)$ is $\varphi^\op$-mean-definable, then $\varphi^{b_m} \rightarrow q$ uniformly in $\bI^A$ and $\hat{b}_m \rightarrow \hat{q} \in \cA\bigl( \cco(\varphi_A) \bigr)$.

    For every $n$ and $m$, we have an equality $\hat{a}_n(\varphi^{b_m}) = \hat{b}_m(\varphi_{a_n})$.
    Since both sequences converge uniformly, $\hat{p}(q) = \hat{q}(p)$ for all $p,q$ under consideration.
    The function $\overline{\varphi}(p,q) = \hat{p}(q) = \hat{q}(p)$ will do.
  \item[\impnum{2}{1}]
    Let $\overline{\varphi}\colon \overline{\varphi_A} \times \overline{\varphi^B} \rightarrow \bI$ denote the separately continuous extension.
    Let $(a_n)$ and $(b_m)$ be as in \autoref{dfn:DoubleLimit}.
    Let $p \in \overline{\varphi_A}$ and $q \in \overline{\varphi^B}$ be accumulation points of $(\varphi_{a_n})$ and $(\varphi^{b_m})$, respectively.
    Then both limits evaluate to $\overline{\varphi}(p,q)$.
  \end{cycprf}
\end{proof}

\begin{cor}
  \label{cor:DoubleLimitWeakEmbedding}
  Let $A$ and $B$ be non-empty sets, and $\varphi\colon A \times B \rightarrow \bI$.
  Then the following are equivalent:
  \begin{enumerate}
  \item
    The function $\varphi$ has the double limit property.
  \item
    \label{item:DoubleLimitWeakEmbeddingClosure}
    There exists a topological embedding $\iota \colon \overline{\varphi_A} \rightarrow C(\overline{\varphi^B})$, in the weak topology on $C(\overline{\varphi^B})$, such that $\varphi(a,b) = \iota(\varphi_a)(\varphi^b)$.
  \item
    \label{item:DoubleLimitWeakEmbeddingCCO}
    There exists an affine topological embedding $\iota\colon \cco(\varphi_A) \rightarrow \cA\bigl( \cco(\varphi^B) \bigr)$, in the weak topology on $\cA(\cco(\varphi^B))$, such that $\varphi(a,b) = \iota(\varphi_a)(\varphi^b)$.
  \end{enumerate}
  Moreover, the embeddings in \autoref{item:DoubleLimitWeakEmbeddingClosure} and \autoref{item:DoubleLimitWeakEmbeddingCCO} are unique and isometric (in the uniform convergence metric on $\bI^B$).
\end{cor}
\begin{proof}
  \begin{cycprf}
  \item[\impnum{1}{3}]
    Use \autoref{thm:DoubleLimitTopological}\autoref{item:DoubleLimitCCO} and \autoref{lem:SeparatelyContinuousSeparatelyAffineWeakTopology}.
  \item[\impprev]
    Since $\overline{\varphi_A} \subseteq \cco(\varphi_A)$ and $\cA\bigl( \cco(\varphi^B) \bigr)$ admits a natural continuous morphism to $C( \overline{\varphi^B} )$.
  \item[\impprev]
    The function $\overline{\varphi}(p,q) = \iota(p)(q)$ is separately continuous, so we may use \autoref{thm:DoubleLimitTopological}\autoref{item:DoubleLimitClosure}.
  \end{cycprf}
\end{proof}

\begin{cor}
  \label{cor:DoubleLimitMazur}
  Let $A$ and $B$ be non-empty sets, and assume that $\varphi\colon A \times B \rightarrow \bI$ has the double limit property.
  Then for every convex subset $C \subseteq \cco(\varphi_A)$, its metric closure and topological closure agree.
\end{cor}
\begin{proof}
  By \autoref{cor:DoubleLimitWeakEmbedding} and \autoref{lem:Mazur}.
\end{proof}

Continuing in the same setting, let $G$ be a group, and assume that it acts on the sets $A$ and $B$ in a manner that leaves $\varphi$ invariant:
\begin{gather*}
  \varphi(ga,gb) = \varphi(a,b) \qquad \forall a \in A, \ b \in B, \ g \in G.
\end{gather*}
The natural action $G \curvearrowright \bI^B$ (respectively, $G \curvearrowright \bI^A$) restricts to an action $G \curvearrowright \cco(\varphi_A)$ (respectively, $G \curvearrowright \cco(\varphi^B)$).

\begin{cor}
  \label{cor:DoubleLimitRyllNardzewski}
  Let $A$ and $B$ be non-empty sets, and assume that $\varphi\colon A \times B \rightarrow \bI$ has the double limit property.
  Let $G$ be a group acting on $A$ and $B$ and leaving $\varphi$ invariant.
  Let $K\subseteq \cco(\varphi_A)$ be a non-empty, closed, convex subset that is $G$-invariant for the induced action $G \curvearrowright \cco(\varphi_A)$.
  Then the restricted action $G \curvearrowright K$ admits a fixed point.
\end{cor}
\begin{proof}
  We are going to use the Ryll-Nardzewski Fixed Point Theorem: if $G$ acts on a weakly compact convex subset of a Banach space $E$ by affine isometries of $E$, then it admits a fixed point (see, for instance, \cite{Namioka:GeometricProof} or \cite[Thm.~5.2]{Glasner:ProximalFlows}).
  By \autoref{cor:DoubleLimitWeakEmbedding}, we may embed $\cco(\varphi_A)$ (and $K$) as a weakly compact convex subset of $E = \cA\bigl( \cco(\varphi^B)\bigr)$.
  The  action $G \curvearrowright \cco(\varphi^B)$ induces an action $G \curvearrowright E$, by affine isometries.

  All that is left to show is that the embedding $\iota\colon \cco(\varphi_A) \rightarrow E$ respects the action of $G$.
  Let $p = \varphi_a$ and $q = \varphi^b$.
  Then $gp = p \circ g^{-1} = \varphi_{ga}$, and similarly, $gq = \varphi^{gb}$.
  Therefore, $\iota(gp)(gq) = \varphi(ga,gb) = \varphi(a,b) = \iota(p)(q)$.
  Since all the actions of $g$ considered here are affine and continuous, we deduce that $\iota(gp)(gq) = \iota(p)(q)$ for every $p \in \cco(\varphi_A)$ and $q \in \cco(\varphi^B)$, so $\iota(gp) = \iota(p) \circ g^{-1} = g \iota(p)$.
  This completes the proof.
\end{proof}

In later sections, $\varphi$ will represent a formula in affine (or, if we wished so, continuous) logic, and the space $\cco(\varphi_A)$ will coincide with the (affine) types in this formula.
Slightly more generally, we shall wish to consider types in several formulas.
We conclude the section by showing that, in the present context, this introduces nothing new.

\begin{lem}
  \label{lem:DoubleLimitGluing}
  Assume that we are given $k \in \bN$, a non-empty set $A$, and for $i < k$, a non-empty set $B_i$ together with a function $\varphi_i\colon A \times B_i \rightarrow \bI$.
  Let $B = \coprod_i B_i$ be the disjoint union, and let $\varphi = \coprod_i \varphi_i \colon A \times B \rightarrow \bI$ (so $\varphi(a,b) = \varphi_i(a,b)$ for $b \in B_i$).

  Then $\varphi$ has the double limit property if and only if each $\varphi_i$ does.
\end{lem}
\begin{proof}
  Every witness to the failure of the double limit property for $\varphi_i$ is also one for $\varphi$.
  Conversely, given a witness $(a_n,b_n)$ to the failure of the double limit property for $\varphi$, there exists $i$ such that $b_n \in B_i$ infinitely often, and we may extract a failure witness for $\varphi_i$.
\end{proof}

Note that in the setting of the previous lemma, in the space $\bI^B = \prod_i \bI^{B_i}$ we have $\varphi_a = \bigl( (\varphi_i)_a : i < k \bigr)$ for every $a\in A$, and thus
\begin{gather*}
  \cco(\varphi_A)
  \subseteq
  \prod_i \cco\bigl( (\varphi_i)_A \bigr)
  \subseteq
  \prod_i \bI^{B_i} = \bI^B.
\end{gather*}

\section{Stability in a structure}
\label{sec:StabilityInStructure}

Let us recall from \cite[\textsection3]{BenYaacov-Ibarlucia-Tsankov:AffineLogic} a few definitions and facts regarding (affine) types in a language $\cL$.
We denote by $\cL^\aff_x$ the collection of all affine formulas in the variables $x$, up to logical equivalence.
It forms an \emph{order unit space}.
If $M$ is an $\cL$-structure and $a \in M^x$, then $p = \tp^\aff(a)$, the \emph{affine type} of $a$ (in $M$), is the map that associates to an affine formula $\varphi \in \cL^\aff_x$ the value $\varphi^p = \varphi(a)$.
Viewed in this fashion, $p$ is a state of $\cL^\aff_x$, and the collection of all affine types (of $x$-tuples, in $\cL$-structures) is exactly the state space $\tS(\cL^\aff_x)$.
We therefore also call it the space of \emph{types} in $\cL$, and denote it $\tS^\aff_x(\cL)$.
Conversely, $\cA\bigl( \tS^\aff_x(\cL) \bigr)$ is the norm-completion of $\cL^\aff_x$.
Types in this sense are \emph{without parameters}.

Affine types \emph{with parameters} in a fixed $\cL$-structure $M$ are defined similarly.
If $N \succeq^\aff M$ is an affine extension and $a \in N^x$, then $p = \tp^\aff(a/M)$, the \emph{affine type over $M$} of $a$, is the map that associates to every affine formula with parameters in $M$, say $\varphi(x,b)$, the value $\varphi(x,b)^p = \varphi(a,b)^N$.
The collection of all such types over $M$ (realised in affine extensions of $M$) is denoted $\tS^\aff_x(M)$.
It is naturally a compact convex set, and the map $p \mapsto \varphi(x,b)^p$ is continuous and affine for every formula $\varphi(x,b)$.
The collection of all such functions is dense in $\cA\bigl( \tS^\aff_x(M) \bigr)$, the order unit space of continuous affine functions on $\tS^\aff_x(M)$, with the norm of uniform convergence.

\begin{dfn}
  \label{dfn:PhiType}
  Let $x$ and $y$ be two disjoint tuples of variables, $a \in M^x$, and $B \subseteq M$.
  Let $\varphi(x,y)$ be a formula in \emph{some} sense: affine, continuous, or anything else that we may interpret as a function on $M^{x,y} = M^x \times M^y$.
  We then define the \emph{$\varphi$-type} of $a$ over $B$, denoted $p = \tp_\varphi(a/B)$, as the map that associates to every \emph{instance} $\varphi(x,b)$, with $b \in B^y$, the value $\varphi(x,b)^p = \varphi(a,b)^M$.
\end{dfn}

Back to the affine context, let us fix an affine $\cL$-formula $\varphi(x,y)$.
If $N \succeq^\aff M$ is an affine extension and $a \in N^x$, then $p = \tp_\varphi(a/M)$, in the sense of the previous paragraph, is an \emph{affine $\varphi$-type} over $M$ (here, \emph{affine} refers to the extension $N \succeq^\aff M$, since $\varphi$ is already assumed affine).
The collection of all affine $\varphi$-types over $M$ is denoted $\tS^\aff_\varphi(M)$.
By definition, $\tp_\varphi(a/M)$ is the restriction of $\tp^\aff(a/M)$ to $M$-instances of $\varphi$.
Therefore, the space $\tS^\aff_\varphi(M)$ is a topological affine quotient of $\tS^\aff_x(M)$, and the functions $p \mapsto \varphi(x,b)^p$, together with the constant functions, generate a dense subspace of $\cA\bigl( \tS^\aff_\varphi(M) \bigr)$.

Let $A = M^x$ and $B = M^y$, so the interpretation of $\varphi$ in $M$ is a function $\varphi\colon A \times B \rightarrow \bI \subseteq \bR$ for some compact interval $\bI$.
An affine $\varphi$-type $p$ is completely determined by the function $dp\colon B \rightarrow \bI$ sending $b \mapsto \varphi(x,b)^p$.
This gives rise to an embedding $\tS^\aff_\varphi(M) \hookrightarrow \bI^B$, and it is not difficult to see that this is an affine topological embedding.
It is moreover isometric, when we equip $\bI^B$ with the metric of uniform convergence, and $\tS^\aff_\varphi(M)$ with the distance function
\begin{gather}
  \label{eq:PhiTypeDistance}
  \partial(p,q) = \|dp - dq\| = \sup_{b \in M^y} \, |\varphi(x,b)^p - \varphi(x,b)^q|.
\end{gather}

We saw in \cite{BenYaacov-Ibarlucia-Tsankov:AffineLogic} that $\tS^\aff_x(M)$ is the closed convex envelope of the set of types realised in $M$.
Therefore, its quotient $\tS^\aff_\varphi(M)$ is the closed convex envelope of the $\varphi$-types realised in $M$, namely, of the set of $\tp_\varphi(a/M)$ for $a \in A^x$.
If $a \in A^x$, then the identification of $\tS^\aff_\varphi(M)$ as a convex compact subset of $\bR^B$ sends $\tp_\varphi(a/M)$ to the function $\varphi_a\colon b \mapsto \varphi(a,b)$.
It follows that $\tS^\aff_\varphi(M)$ as a whole is identified with $\cco(\varphi_A) \subseteq \bR^B$.
Similarly, if $\varphi^\op$ denotes the formula $\varphi^\op(y,x) = \varphi(x,y)$, then $\tS^\aff_{\varphi^\op}(M)$ can be identified with $\cco(\varphi^B) \subseteq \bR^A$.
Thus, we find ourselves in the context of \autoref{sec:DoubleLimit}.

While the case of continuous logic is less pertinent to the present paper, it is worthwhile to mention it as well, albeit more briefly.
We consider arbitrary continuous formulas (rather than only the affine ones), but only consider types realised in elementary extensions $N \succeq^\cont M$.
A \emph{continuous type} $\tp^\cont(a/M)$ can be identified with the map that associates to every continuous formula with parameters in $M$ its value at $a$.
Given a continuous formula $\varphi(x,y)$, an elementary extension $N \succeq^\cont M$ and $a \in N^x$, $\tp_\varphi(a/M)$ is a \emph{continuous $\varphi$-type} over $M$.
The space of all continuous types over $M$ in $x$, denoted $\tS^\cont_x(M)$, is compact, and the space of continuous $\varphi$-types $\tS^\cont_\varphi(M)$ is its topological quotient.
As above, it can be identified with its image in $\bR^B$, which is a compact subset.
Since realised types are dense in continuous logic, $\tS^\cont_\varphi(M)$ is identified with $\overline{\varphi_A} \subseteq \bR^B$, and similarly, $\tS^\cont_{\varphi^\op}(M)$ with $\overline{\varphi^B}$.
In the special case where $\varphi$ is an affine formula, every continuous $\varphi$-type is an affine $\varphi$-type (since every elementary extension is an affine one), which corresponds to the inclusion $\overline{\varphi_A} \subseteq \cco(\varphi_A)$.

Recall from \cite{BenYaacov-Ibarlucia-Tsankov:AffineLogic} that an affine definable predicate on a structure $M$, \emph{with parameters} in $M$, in the variables $x$, is a function $\psi\colon M^x \rightarrow \bR$ that is a uniform limit of affine formulas in $x$ with parameters in $M$.
Such predicates stand in one-to-one correspondence with continuous affine functions on $\tS^\aff_x(M)$, i.e., with members of $\cA\bigl( \tS^\aff_x(M) \bigr)$, and we often tend to ignore the distinction.
In particular, an affine definable predicate on $M$ (with parameters in $M$) extends in a unique fashion to an affine definable predicate on any affine extension $N \succeq^\aff M$ (with parameters in $M \subseteq N$).

\begin{dfn}
  Let $M$ be an $\cL$-structure, and $\varphi(x,y)$ an affine formula.
  Let $\psi\colon M^x \rightarrow \bR$ be a function.
  \begin{itemize}
  \item We say that $\psi$ is an \emph{affine $\varphi$-predicate} on $M$ if it is a uniform limit of affine combinations of $M$-instances of $\varphi$.
  \item We say that $\psi$ is a \emph{mean $\varphi$-predicate} on $M$ if it is a uniform limit of convex combinations of $M$-instances of $\varphi$.
  \end{itemize}
\end{dfn}

An affine $\varphi$-predicate on $M$ is the same thing as an affine definable predicate with parameters in $M$ that factors via $\tS^\aff_\varphi(M)$, i.e., a member of $\cA\bigl( \tS^\aff_\varphi(M) \bigr)$.
In particular, if $\psi$ is an affine $\varphi$-predicate on $M$, then it can be evaluated at any $p \in \tS^\aff_\varphi(M)$.
A mean $\varphi$-predicate is, in particular, an affine $\varphi$-predicate, and can be expressed as a uniform limit of means $\frac{1}{n} \sum_{i<n} \varphi(x,b_i)$.

When $\varphi$ is a continuous formula, we can define a mean $\varphi$-predicate on $M$ in the same manner.
It is in particular a definable $\varphi$-predicate, i.e., a continuous function on $\tS^\cont_\varphi(M)$.

\begin{dfn}
\label{dfn:DefinitionOverModel}
  Let $\varphi(x,y)$ be an affine formula and let $p \in \tS^\aff_\varphi(M)$ be a $\varphi$-type over $M$.
  As before, let $dp\colon M^y\to\bR$ denote the function sending $b \mapsto \varphi(x,b)^p$.
  \begin{itemize}
  \item We say that $p$ is \emph{affinely definable} if $dp$ is an affine definable predicate (in $y$, with parameters in $M$).
    We then call $dp$ the \emph{definition} of $p$, whence the notation.
  \item We say that $p$ is \emph{affinely $\varphi^\op$-definable} if $dp$ is an affine $\varphi^\op$-predicate.
  \item We say that $p$ is \emph{mean-definable} if $dp$ is a mean $\varphi^\op$-predicate.
  \end{itemize}
\end{dfn}

If we allow ourselves to change the language, then we may always make $p$ affinely definable by naming $dp$ by a new predicate symbol.
Being affinely $\varphi^\op$-definable or mean-definable are more robust notions, which are not subject to such ``cheap'' manipulations.

\begin{lem}
  Let $\varphi(x,y)$ be an affine formula, $M$ a structure, $A = M^x$ and $B = M^y$.
  Then the following conditions are equivalent:
  \begin{enumerate}
  \item The function $\varphi\colon A \times B \rightarrow \bR$ factors via a separately continuous and affine function $\overline{\varphi}\colon \cco(\varphi_A) \times \cco(\varphi^B) \rightarrow \bR$.
  \item Every $p \in \tS^\aff_\varphi(M)$ is affinely $\varphi^\op$-definable, every $q \in \tS^\aff_{\varphi^\op}(M)$ is affinely $\varphi$-definable, and moreover, the \emph{symmetry} property holds, namely, $dp(q) = dq(p)$.
  \end{enumerate}
\end{lem}
\begin{proof}
  In one direction, consider $p \in \tS^\aff_\varphi(M) = \cco(\varphi_A)$.
  Then for every $b \in B$, $dp(b) = \overline{\varphi}(p,\varphi^b)$.
  Having fixed $p$, this is a function on $M^y$ that factors continuously and affinely via $\tS^\aff_{\varphi^\op}(M)$, i.e., an affine $\varphi^\op$-predicate.
  Similarly, every $q \in \tS^\aff_{\varphi^\op}(M)$ is affinely $\varphi$-definable.
  Finally, $dp(q) = \overline{\varphi}(p,q) = dq(p)$.

  For the converse, since every $p \in \tS^\aff_\varphi(M)$ is affinely $\varphi^\op$-definable, we may define $\overline{\varphi}_1(p,q) = dp(q)$, and this is affine and continuous in $q$.
  Similarly, we may define $\overline{\varphi}_2(p,q) = dq(p)$, and this is affine and continuous in $p$.
  Finally, by symmetry, the two agree.
\end{proof}

In the continuous case, we have the analogous property, which we state without proof.

\begin{lem}
  \label{lem:separate-continuity-translation-cont}
  Let $\varphi(x,y)$ be a continuous formula, $M$ a structure, $A = M^x$ and $B = M^y$.
  Then the following conditions are equivalent:
  \begin{enumerate}
  \item The function $\varphi\colon A \times B \rightarrow \bR$ factors via a separately continuous function $\overline{\varphi}\colon \overline{\varphi_A} \times \overline{\varphi^B} \rightarrow \bR$.
  \item Every $p \in \tS^\cont_\varphi(M)$ is $\varphi^\op$-definable, every $q \in \tS^\cont_{\varphi^\op}(M)$ is $\varphi$-definable, and moreover, the \emph{symmetry} property holds, namely, $dp(q) = dq(p)$.
  \end{enumerate}
\end{lem}

\begin{dfn}
  Let $M$ be an $\cL$-structure.
  An $\cL$-formula $\varphi(x,y)$ is \emph{stable in $M$} if for every two sequences $(a_n:n\in\bN) \subseteq M^x$ and $(b_n:n\in\bN) \subseteq M^y$, the equality
  \begin{gather*}
    \lim_n \lim_m \varphi(a_n,b_m) = \lim_m \lim_n \varphi(a_n,b_m)
  \end{gather*}
  holds whenever all the limits exist.
\end{dfn}

The following is now an immediate corollary (even, a restatement) of \autoref{thm:DoubleLimitTopological}.
The novelty here is of course the affine case, but even in the continuous setting, the equivalence between stability in a model and separate mean-definability does not seem to have been observed in the literature.

\begin{thm}[Stability in a model and definability of types]
  \label{thm:DoubleLimitFormula}
  Let $M$ be an $\cL$-structure.
  The following are equivalent for a continuous logic $\cL$-formula $\varphi(x,y)$:
  \begin{enumerate}
  \item\label{item:DLF-stable}
    The formula $\varphi$ is stable in $M$.
  \item\label{item:DLF-def-and-sym-cont}
    Every $p \in \tS^\cont_\varphi(M)$ is $\varphi^\op$-definable, every $q \in \tS^\cont_{\varphi^\op}(M)$ is $\varphi$-definable, and moreover, the symmetry property $dp(q) = dq(p)$ holds.
  \item\label{item:DLF-mean-def-cont}
    Every $p \in \tS^\cont_\varphi(M)$ is mean-definable, and every $q \in \tS^\cont_{\varphi^\op}(M)$ is mean-definable.
  \end{enumerate}
  If moreover $\varphi(x,y)$ is an affine formula, then the previous conditions are further equivalent to the following:
  \begin{enumerate}[resume]
  \item\label{item:DLF-def-and-sym-aff}
    Every $p \in \tS^\aff_\varphi(M)$ is affinely $\varphi^\op$-definable, every $q \in \tS^\aff_{\varphi^\op}(M)$ is affinely $\varphi$-definable, and moreover, the symmetry property $dp(q) = dq(p)$ holds.
  \item\label{item:DLF-mean-def-aff}
    Every $p \in \tS^\aff_\varphi(M)$ is mean-definable, and every $q \in \tS^\aff_{\varphi^\op}(M)$ is mean-definable.
  \end{enumerate}
\end{thm}

\begin{rmk}
  The symmetry property is necessary in conditions \autoref{item:DLF-def-and-sym-cont} and \autoref{item:DLF-def-and-sym-aff}.
  For example, in the structure $(\bN,\leq)$, define $\varphi(a,b)=1$ if $a \leq b$ and $\varphi(a,b)=0$ otherwise.
  Then the only non-realised continuous $\varphi$-type over $\bN$ is $\varphi^\op$-definable (by the constant function 0), the only non-realised continuous $\varphi^\op$-type over $\bN$ is $\varphi$-definable (by the constant function 1), and the realised types are definable as well.
  It follows that all affine $\varphi$-types and $\varphi^\op$-types are affinely definable, since every such type can be obtained as a countable convex combination of continuous types.

  Similarly, mean-definability on both sides is necessary in conditions \autoref{item:DLF-mean-def-cont} and~\autoref{item:DLF-mean-def-aff}.
  Indeed, repeat the previous construction in the structure $(\bN\cup\{\infty\},\leq)$.
  Then every continuous $\varphi$-type over $\bN\cup\{\infty\}$ is realised, and in particular mean-definable.
  It follows that every affine $\varphi$-type over $\bN\cup\{\infty\}$ is mean-definable.
\end{rmk}

\begin{rmk}
  In the case of continuous logic (so in particular, in the case of classical \textit{True}/\textit{False} logic), the argument of \autoref{thm:DoubleLimitTopological} offers a very short and simultaneous proof of Shelah's fundamental theorem that for a stable formula, every type over a model is definable, together with Harrington's symmetry lemma (i.e., the implication \autoref{item:DLF-stable} $\Longrightarrow$ \autoref{item:DLF-def-and-sym-cont} in the previous theorem).
  Indeed, given $p \in \tS^\cont_\varphi(M)$ and $q \in \tS^\cont_{\varphi^\op}(M)$, one may expand the language with predicates for $dp$ and $dq$, then use Löwenheim--Skolem to replace $M$ with a separable elementary substructure in the expanded language (or in a countable sublanguage containing $\varphi$, $dp$ and $dq$).
  We may then take \emph{sequences} $a_n \in M^x$ and $b_m \in M^y$ whose types converge to $p$ and to $q$, and set $\overline{\varphi}(p,q) = \lim_n \varphi(a_n,y)^q = \lim_m \varphi(x,b_m)^p$, which is well-defined (by the stability of $\varphi$) and separately continuous.
  We  conclude by the straightforward \autoref{lem:separate-continuity-translation-cont}.
  Note that the argument does not make use of Grothendieck's theorem invoked in \cite{BenYaacov:Grothendieck}.
\end{rmk}

Let us now consider several formulas simultaneously.
First of all, we may extend \autoref{dfn:PhiType}.

\begin{dfn}
  \label{dfn:FamilyPhiType}
  Let $x$ a tuple of variables, and consider a family $\Phi$ of formulas (in the very general sense of \autoref{dfn:PhiType}) $\varphi(x,y_\varphi)$, allowing the variables $y_\varphi$ to vary with $\varphi$.
  Given $a \in M^x$ and $B \subseteq M$, we define the \emph{$\Phi$-type} $p = \tp_\Phi(a/B)$ as the union of all $\tp_\varphi(a/B)$ for $\varphi \in \Phi$.
  In particular, $\tp_{\{\varphi\}}(a/B) = \tp_\varphi(a/B)$.
\end{dfn}

Here, we are going to assume that $\Phi$ consists of affine $\cL$-formulas.
As expected, given an $\cL$-structure $M$, we define $\tS^\aff_\Phi(M)$, the space of \emph{affine $\Phi$-types} over $M$, as the collection of all $\Phi$-types realised in affine extensions of $M$.
As for a single formula, $\tp_\Phi(a/M)$ is the restriction of $\tp^\aff(a/M)$ to $M$-instances of formulas in $\Phi$, so the space $\tS^\aff_\Phi(M)$ is a topological affine quotient of $\tS^\aff_x(M)$, and the functions $p \mapsto \varphi(x,b)^p$, for $\varphi \in \Phi$ and $b \in M^{y_\varphi}$, together with the constant functions, generate a dense subspace of $\cA\bigl( \tS^\aff_\Phi(M) \bigr)$.
When $\Phi = \{\varphi\}$ is a singleton, the space $\tS^\aff_\Phi(M)$ coincides with $\tS^\aff_\varphi(M)$.
When $\Phi$ contains all affine formulas of the form $\varphi(x,y)$, with $x$ fixed (or a generating set of a dense subset), the space $\tS^\aff_\Phi(M)$ coincides with $\tS^\aff_x(M)$.

Expanding on our treatment of a single formula, let $B_\varphi = M^{y_\varphi}$.
If $p = \tp_\Phi(a/M) \in \tS^\aff_\Phi(M)$, then for every $\varphi \in \Phi$ we have a \emph{$\varphi$-definition} $d_\varphi p \colon B_\varphi \rightarrow \bR$, sending $b \mapsto \varphi(a,b) = \varphi(x,b)^p$.
Letting $p$ vary, this gives rise to the pseudometric
\begin{gather}
  \label{eq:FamilyPhiTypeDistance}
  \partial_\varphi(p,q) = \|d_\varphi p - d_\varphi q\| = \sup_{b \in M^{y_\varphi}} \, |\varphi(x,b)^p - \varphi(x,b)^q|.
\end{gather}
When $\Phi = \{\varphi\}$ is a singleton, this is the distance function defined in \autoref{eq:PhiTypeDistance}.
When $\Phi = \{\varphi_i : i < k\}$ is finite, all reasonable manners to combine these into a single distance function, such as $\sum_i \partial_{\varphi_i}$ or $\bigvee_i \partial_{\varphi_i}$, yield equivalent metric structures.

If $\Phi_0 \subseteq \Phi_1$, then the restriction map $\tS^\aff_{\Phi_1}(M) \rightarrow \tS^\aff_{\Phi_0}(M)$ is a continuous affine quotient map.
We may then identify $\tS^\aff_\Phi(M)$ with the inverse limit
\begin{gather*}
  \tS^\aff_\Phi(M) = \projlim_{\Phi_0 \subseteq \Phi \ \text{finite}} \tS^\aff_{\Phi_0}(M).
\end{gather*}

In light of the previous paragraph, let us consider the special case where $\Phi = \bigl\{ \varphi_i(x,y_i) : i < k \bigr\}$.
Let $B_i = M^{y_i}$ and $B = \coprod_i B_i$ (and, as earlier, $A = M^x$).
Since $\Phi$ is finite, we may also find a compact interval $\bI \subseteq \bR$ that contains all possible values of the formulas $\varphi_i$.
The interpretation of each formula is a function $\varphi_i\colon A \times B_i \rightarrow \bI$, which we may join together to a function $\varphi\colon A \times B \rightarrow \bI$, as in \autoref{lem:DoubleLimitGluing}.
In particular, $\varphi$ has the double limit property if and only if each $\varphi_i$ is stable in $M$.
Let us identify $\tS^\aff_{\varphi_i}(M)$ with $\cco\bigl((\varphi_i)_A\bigr)$.
If $a \in A = M^x$, then $\tp_\Phi(a/M)$ is identified with $\varphi_a = \bigl( (\varphi_i)_a : i < k \bigr)$.
Since the convex hull of realised types is dense in all affine types,
\begin{gather}
  \label{eq:SPhi-as-ccophiA}
  \tS^\aff_\Phi(M) = \cco(\varphi_A) \subseteq \prod_i \tS^\aff_{\varphi_i}(M) \subseteq \prod_i \bI^{B_i} = \bI^B.
\end{gather}
This identification is topological and affine.
It is isometric if we equip $\tS^\aff_\Phi(M)$ with the supremum metric $\bigvee_i \partial_{\varphi_i}$.
This, together with \autoref{lem:DoubleLimitGluing}, provides us with the following, essentially rephrasing \autoref{cor:DoubleLimitMazur}:

\begin{cor}
  \label{cor:StableMazur}
  Let $M$ be an $\cL$-structure and $\Phi = \bigl\{ \varphi_i(x,y_i) : i < k \bigr\}$ be a finite family of affine formulas, all stable in $M$.
  Then for every convex subset $C\subseteq \tS^\aff_\Phi(M)$, its topological closure agrees with its metric closure in the supremum distance $\bigvee_i \partial_{\varphi_i}$.
\end{cor}

Let $\Aut(M)$ denote the automorphism group of $M$.
It acts naturally on $\tS^\aff_\Phi(M)$ through the parameters, i.e., $\varphi_i(x,b)^{gp} = \varphi_i(x,g^{-1} b)^p$.

\begin{cor}
  \label{cor:StableRyllNardzewski}
  Let $M$ be an $\cL$-structure and $\Phi$ be a family of affine formulas $\varphi(x,y_\varphi)$, all stable in $M$.
  Let $G \leq \Aut(M)$ be a group of automorphisms, and let $K \subseteq \tS^\aff_\Phi(M)$ be a non-empty compact convex set, invariant under the action of $G$.
  Then the action $G \curvearrowright K$ admits a fixed point.
\end{cor}
\begin{proof}
  When $\Phi$ is finite, this follows from the identification \autoref{eq:SPhi-as-ccophiA}, \autoref{lem:DoubleLimitGluing}, and \autoref{cor:DoubleLimitRyllNardzewski}.
  In the general case, for each finite $\Psi \subseteq \Phi$, the projection $\pi_\Psi\colon \tS^\aff_\Phi(M)\to \tS^\aff_\Psi(M)$ is continuous, affine, and $G$-equivariant.
  Therefore, by the finite case, there exists $p_\Psi \in K$ such that $\pi_\Psi(p_\Psi)$ is a fixed point in $\pi_\Psi(K)$.
  Any cluster point of the net $(p_\Psi : \Psi \subseteq \Phi\ \text{finite})$ is a fixed point for the action of $G$ on~$K$.
\end{proof}

\section{Preservation under direct integrals}
\label{sec:DirectIntegrals}

In what follows, we will show that stability of an affine formula in a structure is preserved under \emph{direct integrals}, as introduced in \cite[\textsection8]{BenYaacov-Ibarlucia-Tsankov:AffineLogic}.
Let us recall the idea of the construction, and refer to the given reference for more details and discussion.

Let $(\Omega,\cB,\mu)$ be a probability space.
A \emph{measurable field} of structures, based on $\Omega$, is a family $M_\Omega = (M_\omega : \omega\in\Omega)$ of $\cL$-structures together with a tuple $e_I$ of sections $e_i\colon \Omega \to \coprod_{\omega\in\Omega}M_\omega$, $i\in I$, satisfying the following properties:
\begin{itemize}
\item The tuple $e_I$ is a \emph{pointwise enumeration} of the field, meaning that $\{e_i(\omega) : i\in I\}$ is dense in $M_\omega$ for every $\omega\in\Omega$.
\item For every predicate symbol $P$ and function symbol $F$, and for every $i\in I$ and every subtuple $\bar e$ of $e_I$ of the appropriate length/sort, the functions $\omega\mapsto P^{M_\omega}(\bar e(\omega))$ and $\omega\mapsto d(F^{M_\omega}(\bar e(\omega)),e_i(\omega))$ are $\cB$-measurable.
\item Given a subset $I_0\subseteq I$, denote $M_{\omega,I_0} = \overline{\{e_i(\omega) : i\in I_0\}} \subseteq M_\omega$; given a sublanguage $\cL_0\subseteq\cL$, denote by $\fI^\aff(M_\Omega,e_I,\cL_0)$ the collection of all countable subsets $I_0\in \cP_{\aleph_0}(I)$ such that $M_{\omega,I_0} \preceq^\aff_{\cL_0} M_\omega$ for $\mu$-almost every $\omega\in\Omega$.
  Then for every finite (equivalently, countable) sublanguage $\cL_0\subseteq\cL$, the collection $\fI^\aff(M_\Omega,e_I,\cL_0)$ is cofinal in $\cP_{\aleph_0}(I)$.
\end{itemize}
Note that if the index set $I$ is countable, then the last point holds automatically.

If $(M_\Omega,e_I)$ is a measurable field, then a section $f\colon\Omega\to \coprod_{\omega\in\Omega}M_\omega$ is a \emph{measurable section} if it is an almost sure pointwise limit of \emph{simple sections}, i.e., sections of the form $\omega\mapsto e_{k(\omega)}(\omega)$ for a measurable function $k\colon\Omega\to I$ with finite range.
A measurable section is $I_0$-measurable for a subset $I_0\subseteq I$ if for almost every $\omega\in\Omega$, $f(\omega)\in M_{\omega,I_0}$.
We denote the collection of all $I_0$-measurable sections, up to equality almost everywhere, by $M_{\Omega,I_0}$.
The set $M_{\Omega,I}$ of all measurable sections is the \emph{direct integral} of the measurable field $(M_\Omega,e_I)$, and we also denote it by
\begin{gather*}
  M_{\Omega,I} = \int_\Omega^\oplus M_\omega d\mu(\omega).
\end{gather*}
The direct integral, with the induced $L^1$ metric, becomes an $\cL$-structure by interpreting the function symbols pointwise and by defining the value of predicate symbols by integration.
Furthermore, as per \cite[Thm.~8.11]{BenYaacov-Ibarlucia-Tsankov:AffineLogic}, if $f$ is any $x$-tuple of measurable sections and $\psi(x)$ is an affine formula, then the map $\omega\mapsto \psi(f(\omega))$ is $\mu$-measurable (i.e., measurable for the $\mu$-completion of $\cB$), and moreover:
\begin{gather}\label{eq:Los}
  \psi^{M_{\Omega,I}}(f) = \int_\Omega \psi^{M_\omega}(f(\omega)) d\mu(\omega).
\end{gather}

When all structures $M_\omega$ of the family $M_\Omega$ are equal to a given structure $M$, there is a canonical choice of pointwise enumeration $e_I$, given by the constant sections.
The direct integral of the corresponding measurable field is the \emph{direct multiple} of $M$ by $(\Omega,\cB,\mu)$, and is denoted by $L^1(\Omega,M)$.
See \cite[Def.~8.15]{BenYaacov-Ibarlucia-Tsankov:AffineLogic}.

\begin{rmk}\label{rmk:(Omega,B',mu)}
  Let $(M_\Omega,e_I)$ be a measurable field of $\cL$-structures based on the probability space $(\Omega,\cB,\mu)$.
  We may denote it more explicitly by $(M_\Omega,\cB,e_I)$, and its direct integral by $M_{\Omega,\cB,I}$.
  Suppose $\cB'\subseteq\cB$ is a $\sigma$-subalgebra of subsets of $\Omega$ such that $(M_\Omega,\cB',e_I)$ remains a measurable field, i.e., for every predicate symbol $P(x)$ and function symbol $F(x)$, and for every $x$-subtuple $\bar e$ of $e_I$ and every $i\in I$, the maps $\omega\mapsto P\big(\bar e(\omega)\big)$ and $\omega\mapsto d\big(e_i,F(\bar e(\omega))\big)$ are $\cB'$-measurable.
  Then the corresponding direct integral $M_{\Omega,\cB',I}$ is an affine substructure of $M_{\Omega,\cB,I}$.
  Indeed, for every $x$-tuple $f$ of $\cB'$-measurable sections and every formula $\psi(x)$, we have as in \autoref{eq:Los},
\begin{gather*}
\psi(f)^{M_{\Omega,\cB',I}} = \int_\Omega \psi(f(\omega)) d\mu(\omega) = \psi(f)^{M_{\Omega,\cB,I}}.
\end{gather*}
\end{rmk}

Let $(M_\Omega,e_I)$ be a measurable field, and let $p_\Omega = (p_\omega : \omega \in \Omega)$, where $p_\omega \in \tS^\aff_x(M_\omega)$.
Say that the family $p_\Omega$ is \emph{measurable} if for every affine $\cL$-formula $\varphi(x,y)$, and every family of sections $b \in M_{\Omega,I}^y$, the function $\omega \mapsto \varphi\bigl(x, b(\omega) \bigr)^{p_\omega}$ is measurable.
If $p \in \tS^\aff_x(M_{\Omega,I})$ and
\begin{gather}
  \label{eq:TypeDisintegration}
  \varphi(x,b)^p = \int_\Omega \varphi\bigl(x, b(\omega) \bigr)^{p_\omega} d\mu(\omega)
\end{gather}
for all $\varphi$ and $b$ as above, then we say that $p = \int_\Omega p_\omega \, d\mu$.
Similarly, for a fixed affine formula $\varphi(x,y)$, families of $p_\omega \in \tS^\aff_\varphi(M_\omega)$ and $p \in \tS^\aff_\varphi(M_{\Omega,I})$.

The following is a disintegration result for affine types over separable direct integrals.

\begin{prp}
  \label{prp:TypeDisintegration}
  Let $\cL$ be a countable language.
  Let $(M_\Omega,e_I)$ be a measurable field of $\cL$-structures, and assume that its direct integral $M_{\Omega,I}$ is separable.
  Then for every affine type $p \in \tS^\aff_x(M_{\Omega,I})$ there exists a measurable family of types $p_\omega \in \tS^\aff_x(M_\omega)$ such that $p = \int_\Omega p_\omega \, d\mu$.
  
  Consequently, if $\varphi(x,y)$ is a fixed affine $\cL$-formula, then for every affine $\varphi$-type $p \in \tS^\aff_\varphi(M_{\Omega,I})$ there exists a measurable family of types $p_\omega \in \tS^\aff_\varphi(M_\omega)$ such that $p = \int_\Omega p_\omega \, d\mu$.
\end{prp}
\begin{proof}
  Let $B = (b_k : k \in \bN)$ enumerate a dense set of sections in $M_{\Omega,I}$.
  Let us write $Y = (y_k : k \in \bN)$ for a corresponding tuple of variables.
  The map $t\colon \Omega \rightarrow \tS^\aff_Y(\cL)$ sending $\omega\mapsto \tp^\aff\bigl( B(\omega) \bigr)$ is measurable, and we may consider the pushforward measure $\tilde{\mu} = t_*\mu \in \cM\bigl( \tS^\aff_Y(\cL) \bigr)$, whose barycentre is $\tp^\aff(B)$.

  Let $a$ be a realisation of $p \in \tS^\aff_x(M_{\Omega,I})$ in an affine extension of $M_{\Omega,I}$.
  By \cite[Lemma~3.6]{BenYaacov-Ibarlucia-Tsankov:AffineLogic}, there exists $\nu\in \cM\bigl( \tS^\aff_{xY}(\cL) \bigr)$ with barycentre $\tp^\aff(aB)$ and such that $(\pi_Y)_*\nu = \tilde{\mu}$, where $\pi_Y\colon \tS^\aff_{xY}(\cL)\to \tS^\aff_Y(\cL)$ is the variable restriction map.
  Let also $\bE[\cdot | \pi_Y] \colon L^\infty\bigl( \tS^\aff_{xY}(\cL), \nu \bigr) \to L^\infty\bigl( \tS^\aff_Y(\cL), \tilde{\mu} \bigr)$ denote the conditional expectation operator induced by $\pi_Y$.

  Let $\cL^\aff_{xY}$ denote the order unit space of affine formulas in the variables $xY$ (in which only finitely many variables actually occur).
  The collection of instances $\varphi(x,B)$, for $\varphi \in \cL^\aff_{xY}$ is dense in the space of all affine formulas with parameters in $M_{\Omega,I}$.
  In addition,
  \begin{gather*}
    \varphi(x,B)^p
    = \int_{\tS^\aff_{xY}(\cL)} \varphi \, d\nu
    = \int_{\tS^\aff_Y(\cL)} \bE[\varphi | \pi_Y] \, d\tilde{\mu}
    = \int_\Omega \bE[\varphi | \pi_Y] \circ t \, d\mu.
  \end{gather*}
  The map $\cL^\aff_{xY} \to L^\infty(\Omega)$ sending $\varphi\mapsto \bE[ \varphi | \pi_Y] \circ t$ is a unital, positive, linear operator.
  Since $Y$ and $\cL$ are countable, the space $\cL^\aff_{xY}$ is separable.
  Therefore, possibly replacing $\Omega$ with a full measure subset, for each formula $\varphi \in \cL^\aff_{xY}$ we can find a measurable representative $P_\varphi\colon \Omega\to \bR$ for $\bE[ \varphi | \pi_Y] \circ t$ such that the map $\varphi \mapsto P_\varphi$ is linear, positive and unital.
  It follows that for every $\omega \in \Omega$, the map
  \begin{gather*}
    q_\omega\colon \cL^\aff_{xY} \rightarrow \bR,\quad \varphi \mapsto P_\varphi(\omega)
  \end{gather*}
  is a unital, positive, linear functional, i.e., a state.
  Moreover, if $\varphi \in \cL^\aff_Y \subseteq \cL^\aff_{xY}$, we can choose $P_\varphi$ to be the measurable function sending $\omega\to \varphi\bigl( B(\omega) \bigr)$.

  Let $\omega \in \Omega$.
  Then $q_\omega\in\tS^\aff_{xY}(\cL)$ is an affine type, and if $\varphi \in \cL^\aff_Y$, then $q_\omega(\varphi) = P_\varphi(\omega) = \varphi\bigl( B(\omega) \bigr)$, so $\pi_Y(q_\omega) = \tp^\aff\bigl( B(\omega) \bigr)$.
  In other words, $q_\omega\bigl( x, B(\omega) \bigr)$ is an affine type in $\tS^\aff_x\bigl( B(\omega) \bigr)$, which we may extend to a type $p_\omega \in\tS^\aff_x(M_\omega)$.

  By construction, for $\varphi \in \cL^\aff_{xY}$:
  \begin{gather*}
    \varphi(x,B)^p
    = \int_\Omega P_\varphi \, d\mu
    = \int_\Omega \varphi(x,Y)^{q_\omega} \, d\mu
    = \int_\Omega \varphi\bigl(x, B(\omega) \bigr)^{p_\omega} \, d\mu.
  \end{gather*}
  Equivalently, \autoref{eq:TypeDisintegration} holds for every affine $\varphi(x,y)$ and every $b \in B^y$.
  Since $B$ is dense in $M_{\Omega,I}$, for every $b\in M_{\Omega,I}^y$ the function $\omega\mapsto\varphi\bigl( x, b(\omega) \bigr)^{p_\omega}$ is $\mu$-measurable, as a pointwise limit of $\mu$-measurable functions (see \cite[Prop.~8.8]{BenYaacov-Ibarlucia-Tsankov:AffineLogic}), and the equation \autoref{eq:TypeDisintegration} of the statement holds.

  For the last part, let $\pi_\varphi \colon \tS^\aff_x(M_\omega) \rightarrow \tS^\aff_\varphi(M_\omega)$ denote the quotient maps.
  Find $\overline{p} \in \tS^\aff_x(M)$ such that $\pi_\varphi(\overline{p}) = p$, apply the main result, and let $p_\omega = \pi_\varphi(\overline{p}_\omega)$.
\end{proof}

\begin{lem}
  \label{lem:MeanDefinableIntegral}
  Let $\varphi(x,y)$ be an affine $\cL$-formula.
  Let $(\Omega,\cB,\mu)$ be a probability space with separable measure algebra, and let $(M_\Omega,e_I)$ be a measurable field of $\cL$-structures with $I$ countable.
  Suppose that for almost every $\omega\in\Omega$, every $\varphi$-type over $M_\omega$ is mean-definable.
  Then every $\varphi$-type over the direct integral $M_{\Omega,I}$ is mean-definable.
\end{lem}
\begin{proof}
  Under the hypotheses, the direct integral $M_{\Omega,I}$ is separable.
  Let $p\in\tS^\aff_\varphi(M_{\Omega,I})$, and disintegrate it as $\int p_\omega \, d\mu$ as per \autoref{prp:TypeDisintegration}.

  Consider all formulas of the form $\eta(\bar{a},y) = \frac{1}{\ell} \sum_{j<\ell} \varphi(a_j,y)$ where the $a_j$ are any $x$-tuples whose coordinates come from the pointwise enumeration $e_I$, i.e., $a_j=(e_{i_{j,1}},\dots,e_{i_{j,|x|}})$.
  (In particular, for every $\omega\in\Omega$, the tuples of the form $a_j(\omega)$ are dense in $M_\omega^x$.)
  There are countable many of these formulas, so we may enumerate them as $\eta_k(\bar{a}^k,y)$.
  Define $\Omega_{\varepsilon,k}$ as the set of those $\omega\in\Omega$ such that, for all $b \in M_{\Omega,I}^y$, we have
  \begin{gather*}
    \bigl| \varphi\big(x,b(\omega)\big)^{p_\omega} - \eta_k\big(\bar{a}^k(\omega),b(\omega)\big) \bigr| \leq \varepsilon.
  \end{gather*}
  By the choice of $p_\omega$, these sets are $\mu$-measurable.
  Moreover, by hypothesis, almost every $p_\omega$ is mean-definable in $M_\omega$, so $\Omega = \bigcup_k \Omega_{\varepsilon,k}$ up to measure zero.
  Therefore, there exists $k_0$ such that $\mu\left(\bigcup_{k<k_0} \Omega_{\varepsilon,k} \right) > 1-\varepsilon$.
  We may assume that for $k < k_0$, $\eta_k(\bar{a}^k,y) = \frac{1}{\ell} \sum_{j<\ell} \varphi(a_{k,j},y)$, with a common $\ell$.
  For $j < \ell$, and fixing some arbitrary tuple of simple sections $a_0\in M_{\Omega,I}^x$, define
  \begin{gather*}
    \tilde{a}_j(\omega) =
    \begin{cases}
      a_{k,j}(\omega) & \text{if}\ \ k < k_0 \ \ \text{and} \ \ \omega \in \Omega_{\varepsilon,k} \setminus \bigcup_{k' < k} \Omega_{\varepsilon,k'} \\
      a_0(\omega) & \text{if}\ \ \omega \notin \bigcup_{k < k_0} \Omega_{\varepsilon,k}.
    \end{cases}
  \end{gather*}
  Then each $\tilde{a}_j$ is a tuple of simple sections, and
  \begin{gather*}
    \left| \varphi(x,b)^p - \frac{1}{\ell} \sum_{j<\ell} \varphi(\tilde{a}_j,b ) \right|
    \leq \varepsilon\bigl( 1 + 2\|\varphi\|_\infty\bigr)
  \end{gather*}
  for all $b \in M_{\Omega,I}^y$.
  This completes the proof.
\end{proof}

\begin{thm}
  \label{thm:stability-direct-integral}
  Let $\cL$ be a language, and let $\varphi(x,y)$ be an affine $\cL$-formula.
  Let $\Omega = (\Omega,\cB,\mu)$ be a probability space, and $(M_\Omega,e_I)$ a measurable field of $\cL$-structures.

  If $\varphi$ is stable in $M_\omega$ for almost every $\omega\in \Omega$, then it is stable in the direct integral $M_{\Omega,I} = \int^\oplus_\Omega M_\omega \, d\mu$.
\end{thm}
\begin{proof}
  Let $\cL_0\subseteq \cL$ be a countable sublanguage containing the symbols that appear in $\varphi$.
  It is enough to show that $\varphi$ is stable in every separable $\cL_0$-affine substructure $N\preceq_{\cL_0}^\aff M_{\Omega,I}$.
  Using the notation of \autoref{rmk:(Omega,B',mu)}, any such $N$ is contained in an $\cL_0$-affine substructure of $M_{\Omega,I}$ of the form $M^0_{\Omega,\cB_0,I_0}$, where $I_0\subseteq I$ is a countable subset, $\cB_0\subseteq \cB$ is a $\sigma$-subalgebra such that the probability algebra of $(\Omega,\cB_0,\mu)$ is separable, $M^0_\Omega$ denotes the field of $\cL_0$-structures $(M_{\omega,I_0} : \omega\in\Omega)$, and the triplet $(M^0_\Omega,\cB_0,e_{I_0})$ forms a measurable field.
  Indeed, \autoref{rmk:(Omega,B',mu)} ensures that $M^0_{\Omega,\cB_0,I_0} \preceq_{\cL_0}^\aff M^0_{\Omega,\cB,I_0} = M_{\Omega,I_0} \preceq^\aff_{\cL_0} M_{\Omega,I}$.  
  Hence, it is enough to see that $\varphi$ is stable in substructures of the form $M^0_{\Omega,\cB_0,I_0}$.

  Fix such $I_0$ and $\cB_0$.
  Since $M_{\omega,I_0}\preceq^\aff M_\omega$ for almost every $\omega\in\Omega$, the formula $\varphi$ is stable in $M_{\omega,I_0}$ for those $\omega$.
  Therefore, by \autoref{lem:MeanDefinableIntegral} and \autoref{thm:DoubleLimitFormula}\autoref{item:DLF-mean-def-aff}, $\varphi$ is stable in $M^0_{\Omega,\cB_0,I_0}$, as desired.
\end{proof}

\section{Stability in a theory}
\label{sec:StabilityInTheory}

\begin{dfn}
  \label{dfn:StabilityInTheory}
  Let $T$ be an affine theory and $\varphi(x,y)$ an affine formula.
  We say that $\varphi$ is \emph{stable in $T$} if it is stable in every model of $T$.

  If each affine formula $\varphi(x,y)$ is stable in $T$, then $T$ is \emph{affinely stable}.
\end{dfn}

The hypothesis that a formula is stable in a theory allows us to use the Compactness Theorem, and deduce various uniformity results.
The following is a typical example of this phenomenon.

\begin{thm}
  \label{thm:StableMeanDefinableUniformSimultaneous}
  Let $T$ be an affine theory, and let $\Phi = \bigl\{ \varphi_i(x,y_i) : i < k \bigr\}$ be a finite set of formulas, all stable in $T$.
  Then for every $\varepsilon > 0$ there exists $n\in\bN$ such that for every $M\vDash T$ and $p\in\tS^\aff_\Phi(M)$, there are $(a^j : j < n) \subseteq  M^x$ such that for every $i<k$ and $b\in M^{y_i}$,
  \begin{gather}
    \label{eq:StableMeanDefinableUniformSimultaneous}
    \left|\varphi_i(x,b)^p - \frac{1}{n}\sum_{j<n}\varphi_i(a^j,b)\right| < \varepsilon.
  \end{gather}
\end{thm}
\begin{proof}
  Assume not.
  Then there exists $\varepsilon > 0$, and for each $n \in \bN$ there exist $M_n \vDash T$ and $p_n \in \tp^\aff_\Phi(M_n)$ that witness failure for $n!$.
  We may assume that $p_n = \tp_\Phi(c_n/M_n)$, where $M_n \preceq^\aff N_n$ and $c_n \in N_n^x$.

  Let $\omega$ be a non-principal ultrafilter on $\bN$, and let $M_\omega$ and $N_\omega$ denote the respective ultraproducts.
  Then $M_\omega \preceq^\aff N_\omega$ are models of $T$.
  Let $c_\omega \in N_\omega^x$ be the class of $(c_n : n \geq 1)$, and let $p = \tp_\Phi(c_\omega/M_\omega)$.
  Then, by \autoref{cor:StableMazur}, there exist $m$ and a sequence $(a^j : j < m) \subseteq M_\omega^x$ such that \autoref{eq:StableMeanDefinableUniformSimultaneous} holds.

  Let $a^j$ be the class of $(a^j_n)$.
  For each $n \geq m$, we may repeat $n!/m$ times the sequence $(a^j_n : j < m)$, to obtain a sequence of length $n!$.
  Then by hypothesis, there exist $i_n$ and $b_n \in M_n^{y_i}$ witnessing that \autoref{eq:StableMeanDefinableUniformSimultaneous} fails for $p_n$.
  There exists a unique $i < k$ such that $A_i = \{n : n \geq m, \, i_n = i\} \in \omega$.
  We may replace $\bN$ with $A_i$, and define $b \in M_\omega^{y_i}$ to be the class of $(b_n : n \in A_i)$.
  By Łoś's Theorem (for ultraproducts, i.e., for continuous logic), \autoref{eq:StableMeanDefinableUniformSimultaneous} fails for $p$, a contradiction.
\end{proof}

We denote the density character of a metric space by $\fd$.
The density character of a language $\cL$ modulo an $\cL$-theory $T$ is $\fd(\cL) = \sup_x\fd(\cL_x^\aff)$, where $x$ ranges over finite tuples of variables, and $\cL_x^\aff$ is endowed with its natural norm (see \cite[Def.~3.22]{BenYaacov-Ibarlucia-Tsankov:AffineLogic}).

\begin{thm}
  \label{thm:stable-in-T-equivalences}
  Let $\varphi(x,y)$ be an affine formula, and $T$ an affine $\cL$-theory.
  Then the following are equivalent:
  \begin{enumerate}
  \item
    The formula $\varphi$ is stable in $T$.
  \item
    \label{item:thm:stable-in-T-equiv:Extremal}
    The formula $\varphi$ is stable in every extremal model of $T$.
  \item
    \label{item:thm:stable-in-T-equiv:UCCO}
    For every $M \vDash T$, every affine $\varphi$-type over $M$ is mean-definable.
    Equivalently, the convex hull of the realised affine $\varphi$-types over $M$ is metrically dense in $\tS^\aff_\varphi(M)$.
  \item
    \label{item:thm:stable-in-T-equiv:UAR}
    For every $M \vDash T$ and atomless standard probability space $\Omega$, the $\varphi$-types over $M$ realised in $L^1(\Omega,M)$ are metrically dense in $\tS^\aff_\varphi(M)$.
  \item
    For every $M \vDash T$, every $\varphi$-type over $M$ is definable.
  \item
    \label{item:thm:stable-in-T-equiv:DensityAll}
    For every $M \vDash T$, the metric density character of $\tS^\aff_\varphi(M)$ is at most that of $M$.
  \item
    \label{item:thm:stable-in-T-equiv:DensityExists}
    There exists $\kappa \geq \fd_T(\cL)$ such that for every $M \vDash T$ with $\fd(M) \leq \kappa$, the metric density character of $\tS^\aff_\varphi(M)$ is at most $\kappa$.
  \end{enumerate}
\end{thm}
\begin{proof}
  \begin{cycprf}
  \item[\eqnext]
    The non-obvious implication follows from \autoref{thm:stability-direct-integral} and the fact that every model of an affine theory embeds affinely into a direct integral of extremal models (see \cite[Cor.~10.3]{BenYaacov-Ibarlucia-Tsankov:AffineLogic}).
  \item[\impnumnum{1}{3}{5}]
    By \autoref{thm:DoubleLimitFormula}.
  \item[\impnum{5}{7}]
    One can take $\kappa = \fd_T(\cL)$, noting that if $\fd(M)\leq \fd_T(\cL)$ then $\fd_T\bigl( \cL(M) \bigr) = \fd_T(\cL)$.
  \item[\impnum{3}{4}]
    Immediate.
  \item[\impnum{4}{6}]
    If $\Omega$ is a standard probability space, the density character of $L^1(\Omega,M)$ is at most that of $M$, and the map $a \mapsto \tp_\varphi(a/M)$ is uniformly continuous (in the metric on $\varphi$-types).
  \item
    Immediate.
  \item[\impfirst]
    This is the usual argument.
    Let $\lambda$ be the least cardinal such that $2^\lambda > \kappa$ (so $\lambda \leq \kappa$).
    Let $\bO = 2^\lambda = \{0,1\}^\lambda$, equipped with the lexicographic order.
    Let $\bO_0 \subseteq \bO$ consist of those sequences that are constantly zero from some point onward.
    Then $|\bO| = 2^\lambda > \kappa$, while $|\bO_0| = |2^{<\lambda}| \leq \lambda \kappa = \kappa$.

    If $\varphi$ is not stable in $T$, then by compactness we may find a model $N \vDash T$, real numbers $r < s$, and families of tuples $(a_i:i \in \bO) \subseteq N^x$ and $(b_j : j \in \bO) \subseteq N^y$ such that $\varphi(a_i,b_j) \leq r$ if $i < j$ and $\varphi(a_i,b_j) \geq s$ otherwise.
    By Downward Löwenheim-Skolem, there exists $M \preceq^\aff N$ such that $b_j \in M^y$ for all $j \in \bO_0$ and $\fd(M)\leq \kappa$.
    For $i \in \bO$, let $p_i = \tp_\varphi(a_i/M)$.
    If $i < j$ in $\bO$, then there exists $k \in \bO_0$ such that $i < k \leq j$, so $\varphi(a_i,b_k) \leq r$ and $\varphi(a_j,b_k) \geq s$.
    Therefore $\partial(p_i,p_j)\geq s-r$ for every distinct $i,j \in \bO$, showing that the density character of $\tS^\aff_\varphi(M)$ is $> \kappa$.
  \end{cycprf}
\end{proof}

We note that if $\Phi = \bigl\{ \varphi_i(x,y_i) : i < k\bigr\}$ is a finite set of affine formulas, all stable in an affine theory $T$, then condition \autoref{item:thm:stable-in-T-equiv:UAR} of \autoref{thm:stable-in-T-equivalences} holds for $\Phi$-types over models of $T$.
In other words,  every type $p\in\tS^\aff_\Phi(M)$ lies in the metric closure of those realised in $L^1(\Omega,M)$ for any atomless probability space $\Omega$ (in the metric $\bigvee_i \partial_{\varphi_i}$, or $\sum_i \partial_{\varphi_i}$, or any equivalent one).
Stated equivalently, for every $\varepsilon > 0$ there exists $a \in L^1(\Omega,M)$ such that $|\varphi_i(x,b)^p - \varphi_i(a,b)| < \varepsilon$ for every $i < k$ and $b \in M^{y_i}$.

Condition \autoref{item:thm:stable-in-T-equiv:Extremal} of \autoref{thm:stable-in-T-equivalences} allows us to connect stability in affine logic with stability in continuous logic.
We recall that the affine part of a continuous logic theory $\bT$, denoted $\bT_\aff$, is the set of affine conditions implied by $\bT$ (see \cite[\textsection13]{BenYaacov-Ibarlucia-Tsankov:AffineLogic}).

\begin{cor}
  \label{cor:preservation-by-affine-part}
  Let $\bT$ be a continuous logic theory and let $\varphi(x,y)$ be an affine formula.
  If $\varphi(x,y)$ is stable in (every model of) $\bT$, then it is stable in $\bT_\aff$.

  In particular, if a continuous logic theory is stable in continuous logic, its affine part is affinely stable.
\end{cor}
\begin{proof}
  By \cite[Prop.~13.4]{BenYaacov-Ibarlucia-Tsankov:AffineLogic}, every extremal model of $\bT_\aff$ embeds affinely into a model of $\bT$.
  Therefore if $\varphi(x,y)$ is stable in $\bT$, it is also stable in every extremal model of $\bT_\aff$.
  It follows from \autoref{thm:stable-in-T-equivalences} that $\varphi(x,y)$ is stable in $\bT_\aff$.
\end{proof}

\begin{qst}
  \label{qst:ContinuousVsAffineStability}
  Can a more direct argument be given for \autoref{cor:preservation-by-affine-part}?

  Let us elaborate on this.
  Let $\varphi(x,y)$ be a continuous formula, and $\varepsilon > 0$.
  Say that $\varphi$ is \emph{$\varepsilon$-stable} in $\bT$ if the following list of conditions cannot be realised jointly in a model of $\bT$, where $i,j,k,\ell$ vary in $\bN$:
  \begin{gather}
    \label{eq:ContinuousVsAffineStabilityCondition}
    \Bigl\{ \varphi(x_i,y_j) \geq \varphi(x_\ell,y_k) + \varepsilon : i < j, \, k < \ell \Bigr\}.
  \end{gather}
  It is easy to check that $\varphi$ is stable in $\bT$, i.e., in all its models, if and only if it is $\varepsilon$-stable in $\bT$ for every $\varepsilon > 0$.

  By a standard application of the Compactness Theorem for continuous logic, $\varphi$ is $\varepsilon$-stable in $\bT$ if and only if there exists $N_\varepsilon \in \bN$ such that
  \begin{gather}
    \label{eq:ContinuousVsAffineStabilityConditionStable}
    \bT \vDash \sup_{x,y} \bigwedge_{i<j < N_\varepsilon, \, k < \ell < N_\varepsilon} \, \Bigl( \varphi(x_i,y_j) - \varphi(x_\ell,y_k)\Bigr) < \varepsilon.
  \end{gather}

  If $\varphi$ is an affine formula, and $T = \bT$ an affine theory, then we may apply the Compactness Theorem for affine logic to \autoref{eq:ContinuousVsAffineStabilityCondition}.
  It follows that $\varphi$ is $\varepsilon$-stable if and only if there exist $K_\varepsilon \in \bN$, together with a family of $(i_{\varepsilon,s},j_{\varepsilon,s},k_{\varepsilon,s},\ell_{\varepsilon,s})$ satisfying $i_{\varepsilon,s} < j_{\varepsilon,s}$ and $k_{\varepsilon,s} < \ell_{\varepsilon,s}$ for $s < K_\varepsilon$, such that
  \begin{gather}
    \label{eq:ContinuousVsAffineStabilityConditionAffine}
    T \vDash \sup_{x,y} \sum_{s < K_\varepsilon} \, \Bigl( \varphi(x_{i_{\varepsilon,s}},y_{j_{\varepsilon,s}}) - \varphi(x_{\ell_{\varepsilon,s}},y_{k_{\varepsilon,s}}) \Bigr) < K_\varepsilon \varepsilon.
  \end{gather}

  Now assume that $\bT$ is a continuous theory, and $\varphi(x,y)$ is an affine formula, stable in $\bT$.
  Then $\varphi$ is stable in $T = \bT_\aff$, by \autoref{cor:preservation-by-affine-part}.
  Does $\varepsilon$-stability in $\bT$ imply $\varepsilon$-stability in $\bT_\aff$?
  More specifically, how can we deduce a condition of the form \autoref{eq:ContinuousVsAffineStabilityConditionAffine} (for $T = \bT_\aff$) from one of the form \autoref{eq:ContinuousVsAffineStabilityConditionStable}?
  Also, can one replace \autoref{eq:ContinuousVsAffineStabilityConditionAffine} with a form closer to \autoref{eq:ContinuousVsAffineStabilityConditionStable}, with the sum taken over all pairs below some index?
\end{qst}

In the converse direction, we recall that to an affine theory $T$ we may associate two distinguished continuous logic theories, namely: the common continuous theory $T_\ext$ of its extremal models (see \cite[\textsection14]{BenYaacov-Ibarlucia-Tsankov:AffineLogic}); and the common continuous theory $T_\crc$ of the direct integrals of models of $T$ over atomless probability spaces (see \cite[\textsection17]{BenYaacov-Ibarlucia-Tsankov:AffineLogic}), called the \emph{convex realisation completion} of $T$.
Concerning the latter we may observe the following.

\begin{cor}
  \label{cor:T-stable-TCR-stable}
  Let $T$ be an affine theory.
  Then $T$ is affinely stable if and only if $T_\crc$ is stable in continuous logic.
\end{cor}
\begin{proof}
  By \cite[Thm.~17.7]{BenYaacov-Ibarlucia-Tsankov:AffineLogic}, $(T_\crc)_\aff \equiv T$, so one implication follows from \autoref{cor:preservation-by-affine-part}.
  Also by \cite[Thm.~17.7]{BenYaacov-Ibarlucia-Tsankov:AffineLogic}, the theory $T_\crc$ has \emph{delta-convex reduction}, i.e., every continuous formula can be approximated by continuous combinations of affine formulas.
  The other implication follows.
\end{proof}

As a special case, we remark that complete, \emph{Poulsen} theories (i.e., whose type spaces $\tS_x^\aff(T)$ are isomorphic to the Poulsen simplex) are affinely stable if and only if they are stable in continuous logic, because complete Poulsen theories are equivalent to their convex realisation completions (see \cite[Thm.~20.4]{BenYaacov-Ibarlucia-Tsankov:AffineLogic}).
The main examples of this kind are given by ergodic theory. See \autoref{exm:PMP} below.

As another special case, which combines the previous two corollaries, we recover the following result from \cite{BenYaacov-Keisler:MetricRandom}.

\begin{cor}
\label{cor:stability-randomisations}
  Let $\bT$ be a continuous logic theory.
  If $\bT$ is stable in continuous logic, then so is its atomless randomisation, $\bT^{\Rand}$.
\end{cor}
\begin{proof}
  Using the notation of \cite{BenYaacov-Ibarlucia-Tsankov:AffineLogic}, and by Theorem~21.9 therein, $\bT^{\Rand}$ is interdefinable with the theory $\bT' = \bigl((\bT^+)_\Bau\bigr)_\crc$.
  Here, $\bT^+$ is the result of enriching $\bT$ with the value sort $[0,1]$ (and no additional structure), and $(\bT^+)_\Bau = \bigl((\bT^+)_\Mor\bigr)_\aff$ denotes its Bauerisation, i.e., the affine part of its Morleyisation.
  If $\bT$ is stable in continuous logic, then so are $\bT^+$ and $(\bT^+)_\Mor$, and hence $(\bT^+)_\Bau$ is affinely stable by \autoref{cor:preservation-by-affine-part}.
  Thus $\bT'$, and consequently $\bT^{\Rand}$, are stable in continuous logic, by \autoref{cor:T-stable-TCR-stable}.
\end{proof}

\begin{cor}
  Let $T$ be an affine theory such that the extreme type spaces $\cE_x(T)$ are closed.
  Then $T$ is affinely stable if and only if $T_\ext$ is stable in continuous logic.
\end{cor}
\begin{proof}
  Since $T = (T_\ext)_\aff$ (see \cite[Lemma~14.2]{BenYaacov-Ibarlucia-Tsankov:AffineLogic}), one implication follows from \autoref{cor:preservation-by-affine-part} (and holds for arbitrary affine theories).
  For the converse, we recall from \cite[Prop.~16.19]{BenYaacov-Ibarlucia-Tsankov:AffineLogic} that if the extreme type spaces of $T$ are closed, then $T_\ext$ has delta-convex reduction.
  As in the proof of \autoref{cor:T-stable-TCR-stable}, the result follows.
\end{proof}

\begin{exm}
  \label{exm:PMP}
  Let $\Gamma$ be a countably infinite amenable group, and consider the affine theory $\FPMP_\Gamma$ of \emph{free} probability measure-preserving (pmp) actions of $\Gamma$, in the language of probability algebras expanded with unary function symbols for the elements of $\Gamma$ (see \cite[\textsection28]{BenYaacov-Ibarlucia-Tsankov:AffineLogic}).
  This is a complete Poulsen theory with affine (and also continuous) quantifier elimination, and indeed the model completion of the theory $\PMP_\Gamma$ of pmp $\Gamma$-systems.
  It is easy to see that quantifier-free formulas are stable in models of $\PMP_\Gamma$, because probability algebras are stable and the expanded language only adds unary function symbols.
  Hence, by quantifier elimination, the theory $\FPMP_\Gamma$ is stable.

  More generally, for an arbitrary countable group $\Gamma$, consider the affine theory $\PMP_\theta$ of those pmp $\Gamma$-systems having a prescribed IRS $\theta$ that is hyperfinite, ergodic, and nowhere of finite index.
  Then $\PMP_\theta$ is a complete Poulsen theory, and it is also stable for similar reasons.
  See \cite[Thm.~28.10]{BenYaacov-Ibarlucia-Tsankov:AffineLogic}.
  
  For further background concerning the theories $\PMP_\Gamma$ and their model completions in continuous logic, see \cite{BenYaacov-Berenstein-Henson-Usvyatsov:NewtonMS, Giraud2019, Berenstein-Ibarlucia-Henson, GoldbringSewardTucker-Drob}
\end{exm}

\section{Non-forking extensions}
\label{sec:NonForking}

Let us fix a language $\cL$ and a tuple of variables $x$.
Given an $\cL$-structure $M$, let us denote by $\overline{\cL_x^\aff}(M)$ the order unit space of affine definable predicates on the variable $x$ with parameters from $M$ (as mentioned earlier, this space can be identified with $\cA\bigl(\tS_x^\aff(M)\bigr)$). If $A\subseteq M$ is any subset, we may also consider the order unit subspace $\overline{\cL_x^\aff}(A)\subseteq \overline{\cL_x^\aff}(M)$ of affine definable predicates with parameters from~$A$, which is the same thing as the set of those $\psi\in \overline{\cL_x^\aff}(M)$ that are invariant under automorphisms of affine extensions of $M$ fixing $A$ pointwise.

On the other hand, let $\Phi$ be a family of affine $\cL$-formulas of the form $\varphi(x,y_\varphi)$, where the variables $y_\varphi$ vary with $\varphi$ and are disjoint from $x$.
Let us denote by $\overline{\cL_\Phi^\aff}(M)$ the collection of \emph{affine $\Phi$-predicates} on $M$, i.e., the uniform limits of affine combinations of $M$-instances of formulas of $\Phi$.
This is an order unit subspace of $\overline{\cL_x^\aff}(M)$. Given $A\subseteq M$, one may also wish to consider the closed unital subspace generated by $A$-instances of formulas of $\Phi$. However, this subspace could be too small, in that it need not contain all $A$-invariant affine $\Phi$-predicates over $M$. This leads us to the following definition.

\begin{dfn}
  \label{dfn:AffinePhiType}
  Let $M$ be an $\cL$-structure and $A \subseteq M$, and let $\Phi$ be as above.
  We define the collection of \emph{affine $\Phi$-predicate over $A$} as
\begin{gather*}
\overline{\cL_\Phi^\aff}(A) = \overline{\cL_x^\aff}(A) \cap \overline{\cL_\Phi^\aff}(M) \subseteq \overline{\cL_x^\aff}(M).
\end{gather*}
Correspondingly, we define the collection of \emph{affine $\Phi$-types over $A$} as the state space
\begin{gather*}
\tS^\aff_\Phi(A) = \tS\left(\overline{\cL_\Phi^\aff}(A)\right).
\end{gather*}
This gives rise to a commutative diagram of surjective continuous affine maps
    \begin{equation}\label{diagram:Phi-type-spaces}
      \begin{tikzcd}
        \tS^\aff_x(M) \arrow{d}{\rho_{M}} \arrow{r}{\pi_A} & \tS^\aff_x(A) \arrow{d}{\rho_{A}} \\
        \tS^\aff_\Phi(M) \arrow{r}{\pi_{\Phi,A}} & \tS^\aff_\Phi(A)
      \end{tikzcd}
    \end{equation}
    where $\tS^\aff_\Phi(M)$ is as defined after \autoref{dfn:FamilyPhiType}.
    If $a \in M^x$, we define $\tp^\aff_\Phi(a/A)$ as the image of $\tp^\aff(a/M)$ (or of $\tp_\Phi(a/M)$, or of $\tp^\aff(a/A)$) in $\tS^\aff_\Phi(A)$.

    When $\Phi = \{\varphi\}$ consists of a single formula, we write $\overline{\cL^\aff_\varphi}(A)$, $\tS^\aff_\varphi(A)$ and $\tp^\aff_\varphi(a/A)$.
\end{dfn}

Notice that $\tp^\aff_\Phi(a/A)$ carries at least as much information as $\tp_\Phi(a/A)$, as defined in \autoref{dfn:FamilyPhiType}, and possibly more. If $A=M$, then clearly the two agree. The following lemma shows that more generally, when $A$ is an affine substructure of $M$, then the two agree as well.

Let us first recall an important fact regarding affine definable predicates and affine extensions.
Assume that $M \preceq^\aff N$, and that $\psi$ is an affine definable predicate on $M$, with parameters in $M$.
Then $\psi$ admits a unique extension to $N$ that is also affine with parameters in $M$, which we also denote by $\psi$.
In other words, we have a natural inclusion $\overline{\cL_x^\aff}(M)\subseteq \overline{\cL_x^\aff}(N)$ (which restricts to an inclusion $\overline{\cL_\Phi^\aff}(M)\subseteq \overline{\cL_\Phi^\aff}(N)$). We recall as well that definable predicates may be combined and quantified over variables, just as formulas. In particular, given $\psi\in \overline{\cL_x^\aff}(M)$ and an affine formula $\chi(x,c)$ with parameters in $M$, then $\sup_x \, \bigl( \psi(x) - \chi(x,c) \bigr)$ evaluates to the same in $M$ and in $N$.

\begin{lem}
  Assume that $A \subseteq M \preceq^\aff N$ and $a \in M^x$, and let $\Phi$ be as above.
  Then $\overline{\cL_\Phi^\aff}(A)$, the space $\tS^\aff_\Phi(A)$, and the type $\tp^\aff_\Phi(a/A)$, are the same when calculated in $M$ or in $N$.

  In particular, the definition of $\tS^\aff_\Phi(M)$, in the sense of \autoref{dfn:AffinePhiType} viewing $M$ as a subset of $N$, agrees with that of \autoref{sec:StabilityInStructure}, viewing $M$ as a structure in its own right.
\end{lem}
\begin{proof}
  By the preceding reminder, $\overline{\cL_x^\aff}(A) \cap \overline{\cL_\Phi^\aff}(M) \subseteq \overline{\cL_x^\aff}(A) \cap \overline{\cL_\Phi^\aff}(N)$, so one direction is clear.
  
  Conversely, let $\psi\in \overline{\cL_x^\aff}(A) \cap \overline{\cL_\Phi^\aff}(N)$.
  For every $\varepsilon > 0$ there exist formulas $\varphi_i\in\Phi$, an affine combination $\chi(x,\bar{y}) = \alpha + \sum_{i<n} \beta_i \varphi_i(x,y_i)$ and $\bar{b} = (b_i : i <n) \in N^{\bar{y}}$ such that $\sup_x \, |\psi(x) - \chi(x,\bar{b})| < \varepsilon$ in $N$.
  Consequently, in $N$,
  \begin{gather*}
    \sup_x \, \bigl( \psi(x) - \chi(x,\bar{b}) \bigr) + \sup_x \,\bigl( \chi(x,\bar{b}) - \psi(x) \bigr) < 2\varepsilon.
  \end{gather*}
  Since the left hand side is an affine definable predicate, and $\psi\in \overline{\cL_x^\aff}(M)$, there exists $c \in M^{\bar{y}}$ such that, in $M$ (and therefore in $N$),
  \begin{gather*}
    \sup_x \, \bigl( \psi(x) - \chi(x,\bar{c}) \bigr) + \sup_x \,\bigl( \chi(x,\bar{c}) - \psi(x) \bigr) < 2\varepsilon.
  \end{gather*}
  It follows that for some $\gamma \in \bR$,
  \begin{gather*}
    \sup_x \, \bigl( \psi(x) - \chi(x,\bar{c}) \bigr) - \gamma < \varepsilon \quad\text{and}
    \quad
    \sup_x \,\bigl( \chi(x,\bar{c}) - \psi(x) \bigr) + \gamma < \varepsilon.
  \end{gather*}
  Since $\chi(x,\bar{c}) + \gamma$ is an affine combination of $M$-instances of formulas of $\Phi$, we conclude that $\psi\in \overline{\cL_\Phi^\aff}(M)$, as desired.

  We have shown that the collection of affine $\Phi$-predicates over $A$ is the same in $M$ and in $N$, so the identification of the corresponding type spaces $\tS_\Phi^\aff(A)$, computed in $M$ or in $N$, follows.
  In particular, the quotient $\tS^\aff_x(A) \rightarrow \tS^\aff_\Phi(A)$ is the same in $M$ and in $N$.
  Since $\tp^\aff(a/A)$ is the same in $M$ and in $N$, so is $\tp^\aff_\varphi(a/A)$.

  This proves the first assertion.
  The second ones follows, letting $A = M$.
\end{proof}

If $A \subseteq B \subseteq M$, this gives rise to a commutative diagram of surjective continuous affine maps
\begin{center}
  \begin{tikzcd}
    \tS^\aff_x(M) \arrow[d] \arrow[r] & \tS^\aff_x(B) \arrow[d] \arrow[r] & \tS^\aff_x(A) \arrow[d] \\
    \tS^\aff_\Phi(M) \arrow[r] & \tS^\aff_\Phi(B) \arrow[r] & \tS^\aff_\Phi(A)
  \end{tikzcd}
\end{center}
The horizontal arrows are the \emph{parameter restriction} maps.
If $p \in \tS^\aff_x(B)$ and $q = p\rest_A \in \tS^\aff_x(A)$ is its restriction to $A$, then $p$ is an \emph{extension} of $q$ to $B$.
Similarly for $p \in \tS^\aff_\Phi(B)$ and $q = p\rest_A \in \tS^\aff_\Phi(A)$.
By the previous lemma, the square involving $A$ and $B$ does not depend on $M$, i.e., it remains the same if we replace $M$ with an affine extension.

Recall that the natural action of $\Aut(M)$ on $\tS^\aff_x(M)$ is given by
\begin{gather*}
  \psi(x,gb)^{gp} = \psi(x,b)^p,
\end{gather*}
where $\psi(x,b)$ varies over all $M$-instances of affine formulas. The action is continuous and affine.
Note that if $g \in \Aut(M/A)$ (i.e., $g$ fixes $A$ pointwise) and $p \in \tS^\aff_x(M)$, then $p\rest_A = (gp)\rest_A$.

If we restrict to instances of formulas in $\Phi$, we obtain the natural action $\Aut(M) \curvearrowright \tS^\aff_\Phi(M)$.
The two actions are equivariant for the quotient map $\tS^\aff_x(M) \rightarrow \tS^\aff_\Phi(M)$.
Therefore, if $g \in \Aut(M/A)$ and $p \in \tS^\aff_\Phi(M)$, then $p\rest_A = (gp)\rest_A$.

\begin{dfn}
  A type $p\in\tS^\aff_\Phi(M)$ is \emph{affinely definable over $A\subseteq M$} if for each formula $\varphi\in\Phi$, the $\varphi$-definition $d_\varphi p\colon M^{y_\varphi} \rightarrow \bR$ is an affine definable predicate with parameters in $A$.
\end{dfn}

\begin{rmk}
  \label{rmk:phi-type-def-over-A}
  If $p\in\tS^\aff_\Phi(M)$ is affinely definable over $A$ and $\varphi\in\Phi$ is stable in $M$, then by \autoref{thm:DoubleLimitFormula}, the definition $d_\varphi p$ belongs to $\overline{\cL^\aff_{\varphi^\op}}(A)$.
\end{rmk}

\begin{thm}
  \label{thm:LocalStationary}
  Let $M$ be an $\cL$-structure.
  Let $\Phi$ be a family of affine $\cL$-formulas of the form $\varphi(x,y_\varphi)$, all stable in $\Th^\aff(M)$.
  Let $A \subseteq M$ and $p \in \tS^\aff_\Phi(A)$.
  Then the following hold:
  \begin{enumerate}
  \item\label{item:LocalStationary-extention} The type $p$ admits a unique extension to $M$ that is affinely definable over $A$. We denote it by $p^M$.
  \item\label{item:LocalStationary-compatibility} If $p = q\rest_\Phi$ is the projection of a type $q\in\tS^\aff_x(A)$, then $p^M$ is consistent with $q$.
  \item\label{item:LocalStationary-preceq-aff} If $M \preceq^\aff N$, then $d_\varphi p^M = d_\varphi p^N$ for every $\varphi\in\Phi$ (in $M$, or equivalently in $N$, both being affine definable predicates over $A$).
  \end{enumerate}
\end{thm}
\begin{proof}
  We first show the existence part of \autoref{item:LocalStationary-extention}, together with the compatibility condition from \autoref{item:LocalStationary-compatibility}.
  For this, we may replace $M$ with an affine extension.
  Therefore, we may assume that $M$ is affinely $\kappa$-saturated and affinely $\kappa$-homogeneous, where $\kappa > |A| + |\cL| + \aleph_0$ (see \cite[Prop.~6.8, Cor.~17.14]{BenYaacov-Ibarlucia-Tsankov:AffineLogic}).
  Let $p = q\rest_\Phi \in \tS^\aff_\Phi(M)$ with $q\in \tS^\aff_x(A)$.

  Let $\pi_A$, $\pi_{\Phi,A}$ and $\rho_M$ denote the quotient maps as in \autoref{diagram:Phi-type-spaces}, and let
  \begin{gather*}
    K = \rho_M\bigl( \pi_A^{-1}(q)\bigr) \subseteq \tS^\aff_\Phi(M).
  \end{gather*}
  It is a non-empty, compact convex set, invariant under the action of $\Aut(M/A)$.
  Since the diagram \autoref{diagram:Phi-type-spaces} commutes, $\pi_{\Phi,A}(K) = \{p\}$.
  By \autoref{cor:StableRyllNardzewski}, the action $\Aut(M/A)\curvearrowright K$ admits a fixed point, which we may denote by $p^M$.
  In other words, $p^M$ is an $\Aut(M/A)$-invariant extension of $p$ that is consistent with $q$.

  Let us fix a formula $\varphi(x,y_\varphi) \in \Phi$.
  By \autoref{thm:DoubleLimitFormula}, the definition $d_\varphi p^M$ is an affine definable predicate over $M$, which can only use countably many parameters.
  Let $B$ consist of $A$ together with all these parameters.
  In particular, $|B| < \kappa$ and $d_\varphi p^M \in \cA\bigl( \tS^\aff_{y_\varphi}(B) \bigr)$.

  Consider two types $s,t \in \tS^\aff_{y_\varphi}(B)$ that have the same restriction to $\tS^\aff_{y_\varphi}(A)$.
  By saturation, there exist $b,c \in M^{y_\varphi}$ that realise $s$ and $t$, respectively.
  Since $b$ and $c$ have the same affine type over $A$, there exists $g \in \Aut(M/A)$ such that $g(b) = c$.
  Therefore,
  \begin{gather*}
    d_\varphi p^M(s) = \varphi(x,b)^{p^M} = \varphi(x,c)^{gp^M} = \varphi(x,c)^{p^M} = d_\varphi p^M(t).
  \end{gather*}
  We deduce that $d_\varphi p^M$ factors via the quotient map $\tS^\aff_{y_\varphi}(B) \rightarrow \tS^\aff_{y_\varphi}(A)$.
  Therefore, $d_\varphi p^M \in \overline{\cL^\aff_{y_\varphi}}(A)$ for every $\varphi\in\Phi$, and $p^M$ is affinely definable over $A$.

  For uniqueness, we no longer assume that $M$ is saturated.
  Let $p^M_i$, for $i = 0,1$, be two extensions of $p$ to $M$, both affinely definable over $A$, and let $\varphi\in\Phi$.
  Let $b \in M^{y_\varphi}$ and $r = \tp^\aff_{\varphi^\op}(b/A)$.
  We have already established that $r$ admits an extension $r^M \in \tS^\aff_{\varphi^\op}(M)$ that is affinely definable over $A$.
  By \autoref{rmk:phi-type-def-over-A}, its definition $dr^M$ factors via $\tS^\aff_\varphi(A)$.
  Similarly, both $d_\varphi p^M_i$ factor via $\tS^\aff_{\varphi^\op}(A)$.
  Therefore, using the symmetry property of \autoref{thm:DoubleLimitFormula}\autoref{item:DLF-def-and-sym-aff},
  \begin{gather*}
    d_\varphi p^M_i(b) = d_\varphi p^M_i(r^M) = dr^M(p^M_i) = dr^M(p).
  \end{gather*}
  Since $\varphi$ and $b$ are arbitrary, $p^M_0 = p^M_1$.

  For \autoref{item:LocalStationary-preceq-aff}, observe that the restriction of $p^N$ to $M$ is affinely definable over $A$ by the affine definable predicates $d_\varphi p^N$, interpreted in $M$.
  Therefore $p^N\rest_M = p^M$, by uniqueness, and $d_\varphi p^N = d_\varphi p^M$ for every $\varphi\in\Phi$ (in $M$, and therefore also in~$N$).
\end{proof}

\begin{dfn}
  \label{dfn:DefinitionOverSet}
  Let $M$ be an $\cL$-structure, $\Phi$ a family of affine $\cL$-formulas of the form $\varphi(x,y_\varphi)$, $A \subseteq M$, and $p \in \tS^\aff_\Phi(A)$.
  
  Assume that $\varphi(x,y_\varphi) \in \Phi$ is stable in $\Th^\aff(M)$.
  Let $p_\varphi \in \tS^\aff_\varphi(A)$ denote the restriction of $p$ to $\varphi$, and let $p_\varphi^M \in \tS^\aff_\varphi(M)$ denote the unique extension that is affinely definable over $A$, as per \autoref{thm:LocalStationary}.
  Its definition will be called the \emph{$\varphi$-definition} of $p$ (or simply its \emph{definition}, if $\Phi = \{\varphi\}$), denoted
  \begin{gather*}
    d_\varphi p = d (p_\varphi^M) \in \overline{\cL^\aff_{\varphi^\op}}(A).
  \end{gather*}
\end{dfn}

The $\varphi$-definition $d_\varphi p$ does not depend on $M$, by \autoref{thm:LocalStationary}\autoref{item:LocalStationary-preceq-aff}.
By construction, it is an affine $\varphi^\op$-predicate over $A$.
It determines $p_\varphi^M$ (which of course also depends on $M$), and therefore determines $p_\varphi$ (independently of $M$).
If $A = M$, then $p_\varphi = p_\varphi^M$, so this definition of $d_\varphi p$ is compatible with the special case of a $\varphi$-type over a structure treated earlier in \autoref{dfn:DefinitionOverModel} and in the discussion preceding \autoref{eq:FamilyPhiTypeDistance}.
If $\Phi$ consists of stable formulas and $p \in \tS^\aff_\Phi(A)$, then $d_\varphi p$ coincides with the $\varphi$-definition of $p^M \in \tS^\aff_\Phi(M)$, by the uniqueness clause of \autoref{thm:LocalStationary}\autoref{item:LocalStationary-extention}.

\begin{lem}
  \label{lem:DefinitionAffineCombination}
  Let $\Phi$ be a collection of affine formulas of the form $\varphi(x,y_\varphi)$ as usual.
  Let $\psi(x,y) = \alpha + \sum \beta_i \varphi_i(x,y_i)$ be an affine combination of formulas $\varphi_i(x,y_i) \in \Phi$.
  Let $T$ be an affine theory, and $A \subseteq M \vDash T$.
  Then the following hold.
  \begin{enumerate}
  \item If $p \in \tS^\aff_\Phi(A)$, then it admits a unique extension $p' \in \tS^\aff_{\Phi'}(A)$, where $\Phi' = \Phi \cup \{\psi\}$.
  \item If each $\varphi_i$ is stable in $T$, then so is $\psi$.
  \item If both previous items hold and we define $d_\psi p = d_\psi p'$, then
    \begin{gather*}
      d_\psi p(y) = \alpha + \sum \beta_i d_{\varphi_i} p(y_i).
    \end{gather*}
  \end{enumerate}
  In other words, we may freely close $\Phi$ under affine combinations, and definitions of types relative to stable formulas respect affine combinations.
\end{lem}
\begin{proof}
  The first item is immediate, and the second follows from the double limit criterion for stability.
  The last item follows from \autoref{thm:LocalStationary} applied to $p'$ and $\Phi'$.
\end{proof}

Suppose that $\varphi(x,y)$ is an affine formula stable in $\Th^\aff(M)$, and that we have $A \subseteq B \subseteq M$.
Let $p \in \tS^\aff_\varphi(B)$, and let $q \in \tS^\aff_\varphi(A)$ be its restriction to $A$.
Then $dp = dq$ if and only if $dp$ is affine over $A$.
Indeed, one direction holds since $dq$ is affine over $A$.
Conversely, if $dp$ is affine over $A$, then $p^M$ is affinely definable over $A$ and extends $q$, so $p^M = q^M$ and $dp = dp^M = dq^M = dq$.

\begin{dfn}
  \label{dfn:NonForkingExtension}
  Let $T$ be an affine theory, and let $\Phi$ be a set of formulas $\varphi(x,y_\varphi)$, all stable in $T$.
  Let $A \subseteq B\subseteq M \vDash T$, $p \in \tS^\aff_\Phi(B)$ and $q = p\rest_A \in \tS^\aff_\Phi(A)$.
  We say that $p$ \emph{does not fork} over $A$, or that $p$ is the \emph{non-forking extension} of $q$ to $B$, in symbols $p = q^B$, if each $\varphi$-definition of $p$ is affine over $A$ (equivalently, if $d_\varphi p = d_\varphi q$ for every $\varphi \in \Phi$).
\end{dfn}

\begin{lem}
  \label{lem:LocalTransitivity}
  Assume that $\Phi$ is a set of formulas stable in $\Th^\aff(M)$, and let $A \subseteq B \subseteq C \subseteq M$.
  Let $p \in \tS^\aff_\Phi(C)$ and $q = p\rest_A \in \tS^\aff_\Phi(A)$.

  Then $q^C = (q^B)^C$.
  Equivalently, $p$ does not fork over $A$ if and only if it does not fork over $B$ and $p\rest_B$ does not fork over $A$.
\end{lem}
\begin{proof}
  Since $dq^C = dq = dq^B = d(q^B)^C$.
\end{proof}

\begin{lem}
  \label{lem:LocalSymmetry}
  Assume that $\varphi(x,y)$ is stable in $\Th^\aff(M)$, and $A \subseteq M$.
  Let $p \in \tS^\aff_\varphi(A)$, and let $q \in \tS^\aff_{\varphi^\op}(A)$.
  Then
  \begin{gather*}
    dp(q) = dq(p).
  \end{gather*}
\end{lem}
\begin{proof}
  First of all, $dp\in\overline{\cL^\aff_{\varphi^\op}}(A)$, so $dp(q)$ makes sense.
  Since $dp = dp^M$, and $q^M$ extends $q$, this is equal to $dp^M(q^M)$.
  Similarly, $dq(p) = dq^M(p^M)$.
  Finally, by the symmetry property (\autoref{thm:DoubleLimitFormula}\autoref{item:DLF-def-and-sym-aff}), $dp^M(q^M) = dq^M(p^M)$.
\end{proof}

\section{Stability and independence}
\label{sec:Independence}

\begin{dfn}
  Let us consider a complete affine theory $T$.
  \begin{itemize}
  \item For a tuple of variables $x$, let $\Phi^\staff_x$ denote the collection of affine formulas $\varphi(x,y)$ that are stable in $T$.
  \item If $A \subseteq M \vDash T$ and $a \in M^x$, then we define $\tp^\staff(a/A) = \tp^\aff_{\Phi^\staff_x}(a)$.
    We call it the \emph{stable (in $T$) affine type} of $a$ over $A$.
  \item Similarly, we define $\tS^\staff_x(A) = \tS^\aff_{\Phi^\staff_x}(A)$, the space of stable affine types over $A$.
  \end{itemize}
\end{dfn}

\begin{dfn}
  Let $M$ be a structure and let $A,B,C \subseteq M$ be subsets, whose unions we denote by juxtaposition (i.e., $AB = A \cup B$ and so on).
  We say that $A$ and $C$ are \emph{stably affinely independent} over $B$, in symbols $A \ind[\staff]_B\ \  C$, or just $A \ind_B C$ in the present context, if for every finite tuple $a \in A^x$, the stable affine type $\tp^\staff(a/BC)$ does not fork over $B$ (in the sense of $\Th^\aff(M)$).
\end{dfn}

Observe that $A \ind_B C$ if and only if under some (any) enumeration of $A$, the type $\tp^\staff(A/BC)$ does not fork over $B$.
Yet equivalently, if and only if for every pair of finite tuples $a \in A^x$ and $c \in (BC)^y$, and every stable (in $T$) affine formula $\varphi(x,y)$:
\begin{gather*}
  \varphi(a,c) = d_\varphi p(c) = d_\varphi p(q),
\end{gather*}
where $p = \tp^\staff(a/B)$ and $q = \tp^\staff(c/B)$ (or even just $\tp_\varphi(a/B)$ and $\tp_{\varphi^{\op}}(c/B)$, respectively).

The asymmetry in the role of the basis $B$ (added to $C$, but not to $A$) is inevitable if we want to have transitivity in \autoref{thm:StableIndependence}.
When $T$ is stable, this asymmetry disappears.

\begin{lem}
  \label{lem:FullStabilityLeftBasis}
  Assume that the ambient affine theory is stable.
  Then $A \ind_B C$ if and only if $AB \ind_B C$.
\end{lem}
\begin{proof}
  Assume that $A \ind_B C$.
  Let $a \in A^x$ and $b \in B^y$.
  Let $p = \tp^\aff(ab/B)$, $q = p\rest_x = \tp^\aff(a/B)$, and let $a'b'$ realise $p^{BC}$.
  Then $a'$ realises $q^{BC}$, and $d(b',b) = d(y,b)^p = 0$, so $b' = b$.
  On the other hand, $a \vDash q^{BC}$ by hypothesis, so $a \equiv^\aff_{BC} a'$, and therefore $ab \equiv^\aff_{BC} a'b$.
  We conclude that $ab \vDash p^{BC}$.
  Therefore, $AB \ind_B C$.
  The opposite implication is immediate.
\end{proof}

\begin{thm}
  \label{thm:StableIndependence}
  Let $T$ be a complete affine theory.
  The stable affine independence relation $\ind$ on subsets of models of $T$ possesses the following properties.
  \begin{itemize}
  \item \emph{Invariance:} the relation $A \ind_B C$ only depends on the affine type $\tp^\aff(A,B,C)$ (under any enumeration of the sets).
    In particular, the relation is invariant under automorphisms, and if $A,B,C \subseteq M \preceq^\aff N$, then $A \ind_B C$ if and only if $A \ind_B C$ in $N$.
  \item \emph{Finite character:} $A \ind_B C$ if and only if $A_0 \ind_B C$ for every finite subset $A_0 \subseteq A$.
  \item \emph{Conditional symmetry:} if $B \subseteq A \cap C$, then $A \ind_B C$ if and only if $C \ind_B A$.
  \item \emph{Transitivity:} $A \ind_B CD$ if and only if $A\ind_B C$ and $A \ind_{BC} D$.
  \item \emph{Extension:} for every $A,B,C \subseteq M$ there exists an affine extension $N \succeq^\aff M$ and $A' \subseteq N$ such that $A \equiv^\aff_B A'$ and $A' \ind_B C$.
  \item \emph{Local character:} for every $A,C \subseteq M$ there exists a subset $B \subseteq C$ such that $|B| \leq |A| + \fd_T(\cL)$ and $A \ind_B C$.
  \item \emph{Stable stationarity:} if $A \ind_B C$, $A' \ind_B C$ and $A \equiv^\staff_B A'$ (under some enumeration of $A$ and $A'$), then $A \equiv^\staff_{BC} A'$ (under the same enumeration).
  \end{itemize}
  If $T$ is stable, then the independence relation satisfies stronger versions of the symmetry and stationarity axioms:
  \begin{itemize}
  \item \emph{Symmetry:} $A \ind_B C$ if and only if $C \ind_B A$.
  \item \emph{Stationarity:} if $A \ind_B C$, $A' \ind_B C$ and $A \equiv^\aff_B A'$ (under some enumeration of $A$ and $A'$), then $A \equiv^\aff_{BC} A'$ (under the same enumeration).
  \end{itemize}

  Conversely, assume that $\ind[\prime]$ is a relation on subsets of models of $T$ that satisfies all of the above, including full symmetry and stationarity.
  Then $T$ is stable, and $\ind[\prime]$ coincides with $\ind$.
\end{thm}
\begin{proof}
  The proof is entirely standard, and we only include it for the sake of completeness.

  Invariance and the finite character are immediate from the definition.
  Conditional symmetry follows from \autoref{lem:LocalSymmetry}.
  Transitivity follows from \autoref{lem:LocalTransitivity}.
  For extension, let $p = \tp^\staff(A/B)$, and let $A'$ realise $p^{BC}$.
  For the local character, let $p = \tp^\staff(A/C)$, and choose a set of parameters sufficient to define $d_\varphi p$ for every stable $\varphi$.
  Stable stationarity follows from the uniqueness clause of \autoref{thm:LocalStationary}\autoref{item:LocalStationary-extention}.
  If all affine formulas are stable in $T$, then this becomes full stationarity, and full symmetry follows by \autoref{lem:FullStabilityLeftBasis}.

  For the converse, let us fix $\kappa = \fd_T(\cL)$.
  Let $M \vDash T$ have density character at most $2^\kappa$.
  If $M_0 \subseteq M$ is dense and $|M_0| \leq 2^\kappa$, then $|M| \leq |M_0|^{\aleph_0} \leq 2^\kappa$.
  Let $x$ be a finite tuple of variables, and $p \in \tS^\aff_x(M)$, which we may realise by $a$ in some affine extension.
  By local character, there exists a subset $B \subseteq M$ such that $|B| \leq \kappa$ and $a \ind[\prime]_B M$.
  By stationarity, $p$ is uniquely determined by its restriction to $B$.
  There are at most $|M|^\kappa \leq 2^\kappa$ possible subsets $B$, and for each $B$, there are at most $|\bR|^{\fd\bigl( \cL_x^\aff(B) \bigr)} \leq 2^\kappa$ possibilities for $p\rest_B$.
  Therefore, $|\tS^\aff_x(M)| \leq 2^\kappa$.
  By \autoref{thm:stable-in-T-equivalences}\autoref{item:thm:stable-in-T-equiv:DensityExists}, every affine formula $\varphi(x,y)$ is stable in $T$.
  Since $x$ was arbitrary, $T$ is stable.

  In order to show that $\ind = \ind[\prime]$, let us assume that $a \ind_B c$, where $a$ and $c$ are finite tuples.
  By $\ind[\prime]$-extension, in a sufficiently saturated affine extension, we may construct a sequence $(c_i : i < \kappa^+)$ such that $c_i \equiv^\aff_B c$ and $c_i \ind[\prime]_B c_{<i}$.
  By $\ind$-extension, we may find $a' \equiv^\aff_B a$ such that $a' \ind_B c_{<\kappa^+}$.
  Moving the family $c_{<\kappa^+}$ (and using invariance), we may even assume that $a' = a$.
  In particular, $a \ind_B c_i$ for every $i < \kappa^+$ (by $\ind$-transitivity).
  By $\ind$-stationarity and $\ind$-symmetry: $c \equiv^\aff_{Ba} c_i$ for every $i < \kappa^+$.
  On the other hand, by the $\ind[\prime]$-local character, there exists a subset $C_0 \subseteq Bc_{<\kappa^+}$ such that $|C_0| \leq \kappa$ and $a \ind[\prime]_{C_0} B c_{<\kappa^+}$.
  Using $\ind[\prime]$-transitivity, there exists $i < \kappa$ such that $a \ind[\prime]_{B c_{<i}} c_{< \kappa^+}$.
  Using various properties of $\ind[\prime]$, we deduce that $a \ind[\prime]_{B c_{<i}} c_i$, $c_i \ind[\prime]_{Bc_{<i}} a$, $c_i \ind[\prime]_B a c_{<i}$, $c_i \ind[\prime]_B a$, and finally, $a \ind[\prime]_B c_i$.
  Since $c \equiv^\aff_{Ba} c_i$, and by $\ind[\prime]$-invariance: $a \ind[\prime]_B c$.
  The exact same argument works the other way round, and by the finite character, this is enough.
\end{proof}

The \emph{affine definable closure} of a subset $B\subseteq M$ is the set of $a\in M$ such that $\{a\}$ is an affinely definable set (equivalently, the distance function $d(a,y)$ belongs to $\overline{\cL^\aff_y}(B)$).
We denote it by $\dcl^\aff(B)$.

The following is also completely standard, though in the standard case one needs to employ the algebraic closure.

\begin{lem}
  \label{lem:indep-antireflexivity}
  Let $A,B,C$ be subsets of a structure $M$, and suppose moreover that $A$ and $C$ are contained in a sort with a stable metric.
  If $A\ind_B C$, then $A\cap C\subseteq \dcl^\aff(B)$.
\end{lem}
\begin{proof}
  Suppose $A\ind_B C$ and let $a\in A\cap C$, $\varphi(x,y) = d(x,y)$ (the distance in the sort of $a$) and $p = \tp^\aff(a/BC)$. If $a'$ is a realisation of $p^M$ in an affine extension of $M$, then $\varphi(a',a) = d_\varphi p(a) = \varphi(a,a) = 0$, so $a'=a$. Hence $d(a,y) = d_\varphi p(y)$ is a definable predicate with parameters in $B$.
\end{proof}

Let us now point out that non-forking in a stable affine theory $T$ corresponds to non-forking in continuous logic with respect to its convex realisation completion~$T_\crc$.

\begin{prp}
  \label{prp:non-forking-cont-vs-aff}
  Let $T$ be an affinely stable theory, and let $M$ be a model of $T_\crc$.
  Let $A\subseteq B\subseteq M$ and let $a\in M^x$.
  Then $\tp^\aff(a/B)$ does not fork over $A$, in the sense of \autoref{dfn:NonForkingExtension}, if and only if $\tp^\cont(a/B)$ does not fork over $A$, in the sense of continuous logic.
\end{prp}
\begin{proof}
  Possibly passing to an elementary extension, we may assume that $M$ is highly saturated and homogeneous in continuous logic.
  On the other hand, the canonical restriction map $\rho^\aff\colon \tS_x^\cont(M) \to \tS_x^\aff(M)$ is a homeomorphism (see \cite[Cor.~17.11]{BenYaacov-Ibarlucia-Tsankov:AffineLogic}).
  It is also equivariant for the action of $\Aut(M/A)$.

  Let $p = \tp^\aff(a/B)$, and let $q\in \tS_x^\cont(M)$ be the unique type with $\rho^\aff(q) = p^M$.
  If $p = \tp^\aff(a/B)$ does not fork over $A$ in affine logic, then $p^M$ is affinely definable over $A$, and hence $p^M$ and $q$ are invariant under $\Aut(M/A)$.
  By stability of $T_\crc$ (\autoref{cor:T-stable-TCR-stable}) and saturation and homogeneity of $M$, $q$ is definable over $A$ in continuous logic.
  Therefore, $q\rest_B = \tp^\cont(a/B)$ does not fork over $A$, in the sense of $T_\crc$.

  Conversely, suppose that $\tp^\cont(a/B)$ does not fork over $A$ in continuous logic.
  Note that $q\rest_B = \tp^\cont(a/B)$, because the restriction map $\tS_x^\cont(B) \to \tS_x^\aff(B)$ is also a homeomorphism (in particular, injective).
  On the other hand, by the same argument as above, $q$ is non-forking over $B$.
  By transitivity, it is also non-forking over $A$.
  Now consider $r = \tp^\aff(a/A)$ and let $q'\in\tS_x^\cont(M)$ with $\rho^\aff(q') = r^M$.
  Again as before, $q'$ is non-forking over $A$, and in addition $q'\rest_A = \tp^\cont(a/A) = q\rest_A$.
  Since $q$ and $q'$ are two non-forking extensions in continuous logic of a type over $A$, they must be conjugate by an automorphism in $\Aut(M/A)$.
  This automorphism sends $p^M$ to $r^M$, but as the latter is definable over $A$, we have $p^M = r^M$.
  We conclude that $p$ does not fork over $A$, in affine logic.
\end{proof}

\section{A remark regarding Lascar types}
\label{sec:LascarType}

We recall that in either classical or continuous logic, equality of \emph{Lascar strong types} over a set $A$ is the finest bounded $A$-invariant equivalence relation; equivalently, it is the transitive closure of equality of types over models that contain $A$.
It follows easily from either definition that if $a$ is algebraic over $A$, then $a$ is the unique realisation of its Lascar strong type over $A$.

From \autoref{prp:non-forking-cont-vs-aff} and \autoref{lem:indep-antireflexivity} it follows that if $T$ is an affinely stable theory, then in the continuous logic theory $T_\crc$, algebraic and definable closure coincide in the basic sorts (though not necessarily in quotient sorts interpretable in $T_\crc$).
This may be compared with \cite[Cor.~5.9]{BenYaacov:RandomVariables}, where this is proved for atomless randomisations, but with no stability hypothesis.
In fact, in \cite{BenYaacov:RandomVariables} this corollary is deduced from the stronger result that in arbitrary atomless randomisations, every type over a set (in the basic sorts) is Lascar strong.
We can also establish this, in full generality, for affine theories and their convex realisation completions.

\begin{thm}
  Let $T$ be a complete affine theory and $M \vDash T$.
  Let $A \subseteq B \subseteq M$, and let $a,b\in M^x$ be such that $a\equiv^\aff_A b$.
  Then there exists an affine extension $N \succeq^\aff M$ and subset $B' \subseteq N$ such that $B' \equiv^\aff_A B$ (relative to some common enumeration) and $a \equiv^\aff_{B'} b$.
  In particular, $A \subseteq B'$.
\end{thm}
\begin{proof}
  By compactness, it is enough to prove this when $B = Ac$, for some finite tuple $c \in M^y$.
  For this, let us show that the set of conditions
  \begin{gather*}
    \bigl\{ \psi(y) \leq \psi(c) : \psi \in \cL^\aff_x(A) \bigr\}
    \cup
    \bigl\{ \varphi(a,y)\leq \varphi(b,y) : \varphi\in\cL^\aff_{xy}(A) \bigr\}
  \end{gather*}
  is consistent.
  If not, then by the Compactness Theorem for Affine Logic, there exist formulas $\varphi(x,y),\psi(y) \in \cL^\aff(A)$ such that $\psi(c) = 0$, and
  \begin{gather*}
    \inf_y \, \bigl( \varphi(a,y) - \varphi(b,y) + \psi(y) \bigr) = 1.
  \end{gather*}
  Construct a sequence $(a_n : n \in \bN)$ such that $a_0 = a$, $a_1 = b$, and $a_{n+2} a_{n+1} \equiv^\aff_A a_{n+1} a_n$.
  Then $\varphi(a_0,c) - \varphi(a_n,c) \geq n$ for all $n$, which is absurd.
\end{proof}

\begin{cor}
  Let $T$ be a complete affine theory and $M \vDash T_\crc$.
  Let $A \subseteq B \subseteq M$, and let $a,b\in M^x$ such that $a \equiv^\aff_A b$.
  Then there exists an elementary extension $N \succeq^\cont M$ and subset $B' \subseteq N$ such that $B' \equiv^\cont_A B$ (relative to some common enumeration) and $a \equiv^\cont_{B'} b$.

  In particular, $a$ and $b$ have the same elementary type over a model (i.e., an elementary substructure of an ambient model) containing $A$, and therefore, the same Lascar type over~$A$.
\end{cor}
\begin{proof}
  The first assertion holds since every model of $T$ admits an affine extension to a model of $T_\crc$, and in models of $T_\crc$, affine types determine continuous types.
  The second assertion follows by taking $B = M$.
\end{proof}

Notice that this generalises and improves \cite[Thm.~5.8]{BenYaacov:RandomVariables}.
Indeed, it applies to any continuous theory of the form $T_\crc$, which include atomless randomisations as well as other theories, such as complete affine Poulsen theories.
In addition, it shortens the Lascar distance relative to models, in the sense of \cite[Def.~5.1]{BenYaacov:RandomVariables}, from two to one.

\begin{cor}
  Let $T$ be an affine theory and $M$ be a model of $T_\crc$.
  Then for every $A\subseteq M$, $\dcl^\cont(A) = \acl^\cont(A)$.
\end{cor}

Of course, the equality $\dcl = \acl$ also holds in every imaginary/interpretable sort of $M$, in the sense of continuous logic, provided that $A$ only contains members of sorts interpretable in $T$, i.e., in the sense of affine logic.

\bibliographystyle{amsalpha}
\bibliography{bib}
\end{document}